\documentclass[12pt,reqno]{article}

\usepackage{xcolor}
\usepackage{amsmath}
\usepackage{amsthm}
\usepackage{tcolorbox}

\usepackage[colorlinks=true,
linkcolor=webgreen,
filecolor=webbrown,
citecolor=webgreen]{hyperref}

\definecolor{webgreen}{rgb}{0,.5,0}
\definecolor{webbrown}{rgb}{.6,0,0}

\usepackage{color}
\usepackage{fullpage}
\usepackage{float}
\usepackage{graphics}
\usepackage{latexsym}
\usepackage{epsf}
\usepackage{tikz}
\usepackage{tkz-graph}
\usepackage{caption}
\usepackage{mathrsfs}
\usepackage[dvipsnames]{xcolor}
\usepackage{tcolorbox}

\usetikzlibrary{arrows}
\usetikzlibrary{shapes}
\usetikzlibrary{automata}

\newcommand{\seqnum}[1]{\href{https://oeis.org/#1}{\rm \underline{#1}}}

\begin{document}

\theoremstyle{plain}
\newtheorem{theorem}{Theorem}
\newtheorem{corollary}[theorem]{Corollary}
\newtheorem{lemma}[theorem]{Lemma}
\newtheorem{proposition}[theorem]{Proposition}

\theoremstyle{definition}
\newtheorem{definition}{Definition}
\newtheorem{example}{Example}
\newtheorem{conjecture}[theorem]{Conjecture}

\theoremstyle{remark}
\newtheorem{remark}[theorem]{Remark}

\begin{center}
\epsfxsize=4in
\end{center}

\begin{center}
\vskip 1cm{\Large\bf Utilizing Symmetry in Finding New Permutiples from Known Examples\\
\vskip .1in
}
\vskip 1cm
\large
Benjamin V. Holt\\
Department of Mathematics\\
Southwestern Oregon Community College \\
Coos Bay, Oregon 97420\\
USA\\
\href{mailto: benjamin.holt@socc.edu}{\tt benjamin.holt@socc.edu} \\
\end{center}

\vskip .2 in

\begin{abstract}
A permutiple is a natural number whose representation in some base, $b>1$, is an integer multiple of a number whose base-$b$ representation has the same collection of digits. Previous efforts have made progress in finding such numbers using graph-theoretical and finite-state-machine constructions. These are the mother graph and the Hoey-Sloane machine. In this paper, we leverage the inherent symmetry of the above constructions for the purpose of finding new permutiples from old. Such results also help us to see previous work through a new lens.
\end{abstract}

\section{Introduction}
A {\it permutiple} is a number which is an integer multiple of some permutation of its digits in a natural-number base, $b>1$ \cite{holt_3}. We may also describe permutiples as the result of a digit-preserving multiplication. The reader may find other descriptions of equivalent and similar notions in the OEIS \cite{sloane_2}, for instance, ``numbers whose digits can be permuted to get a proper divisor'' (\seqnum{A096093}), or ``numbers $k$ such that $k$ and $nk$ are anagrams'' (\seqnum{A023086}, \seqnum{A023087}, \seqnum{A023088}, \seqnum{A023089}, \seqnum{A023090}, \seqnum{A023091}, \seqnum{A023092}, \seqnum{A023093}). Specific, well-studied cases of permutiple numbers include {\it cyclic numbers}, {\it transposable numbers} \cite{guttman,kalman}, and {\it parasitic numbers} \cite[\seqnum{A092697}]{sloane_2}, all of which involve cyclic permutations of their digits. A base-$10$ example of such a number is $714285 = 5 \cdot 142857$. {\it Palintiple} numbers \cite{hoey, holt_1,holt_2} (\seqnum{A031877}, \seqnum{A222814}, \seqnum{A222815}), another well-studied case, also known as {\it reversal numbers} \cite[\seqnum{A031877}]{sloane_2}, are multiples of their digit reversals. Similarly-defined cases of this phenomenon include {\it reverse multiples} \cite{kendrick_1,sloane, young_1,young_2} (\seqnum{A001232}, \seqnum{A008918}, \seqnum{A008919}), which are also known as {\it reverse divisors} \cite{web_wil}. For the present effort, we retain the use of the term ``palintiple'' to be consistent with the terminology used by Hoey \cite{hoey} and other work  \cite{holt_1,holt_2,holt_3,holt_4,holt_5,holt_6}. The large variety of palintiple (reverse-multiple) types can be organized using a graph-theoretical construction by Sloane \cite{sloane} called {\it Young graphs}, which are a modification of the work of Young \cite{young_1,young_2}. The most widely known examples of palintiples, also in base 10, include $87912=4\cdot 21978$ (\seqnum{A222815}) and $98901 = 9 \cdot 10989$ (\seqnum{A222814}).

Efforts beyond cyclic and reversal permutations include a paper by Qu and Curran \cite{qu} who examined the case of permutiples which are multiples of $(b^{b-1}-1)/(b-1)^2$. In base $10$, these are multiples of $123456789$ (a finite subsequence of \seqnum{A053654}), and two base-$10$ examples include $493827156=4 \cdot 123456789$ and $987654312=8 \cdot 123456789$. Other work \cite{holt_3,holt_4} detailed elementary methods for finding new permutiple examples from the digits of known examples, such as $87912=4\cdot 21978$, to obtain new examples, such as $79128=4\cdot 19782$ and $78912=4\cdot 19728$, which are all multiples of elements in \seqnum{A023088}. This work \cite{holt_3} also found all $5$-digit, base-$6$ permutiples with multiplier $2$ having the same digits as the example $43512=2 \cdot 21534$, each of which is a multiple of an element of \seqnum{A023064}.

Hoey \cite{hoey} used methods from formal language theory, namely, a finite-state-machine construction, to find and characterize all base-$10$ palintiples. In this same work, Hoey also constructed machines which recognize palintiples in other bases, while leaving their general properties as open questions. The state diagrams of these machines bear strong resemblance to Sloane's \cite{sloane} construction, and Sloane acknowledges the connection between Young graphs and formal language theory. Faber and Grantham \cite{faber} also used finite-state-machine techniques to find pairs of integers whose sum is the reverse of their product in an arbitrary base (\seqnum{A360518}). A base-$10$ example of such a pair is $3$ and $24$ since $3+24=27$ and $3\cdot 24=72$. Recent work by the author \cite{holt_5, holt_6} used a finite-state-machine construction and its state graph, called the {\it Hoey-Sloane graph}, to find permutiples. This graph describes single-digit multiplication in a chosen base and multiplier, where each carry is less than the multiplier. The states of the machine are the possible carries of a digit-preserving multiplication, and the input alphabet consists of ordered pairs from the so-called {\it mother graph}, which describes how digits may be permuted in a single-digit multiplication. Digit-preserving multiplications are represented by input strings consisting of certain multiset combinations of mother-graph cycles which enable walks on the Hoey-Sloane graph beginning and ending with the zero state. Several base-$10$ examples with multiplier $4$ were provided (all multiples of elements in \seqnum{A023088}), as well as descriptions of all base-$4$ permutiples (\seqnum{A023059}, \seqnum{A023060}).

In this article, we draw particular attention to the reflective symmetry of the mother graph, and we detail how the Hoey-Sloane graph inherits this symmetry. From these investigations, we see how notions of symmetry arise when finding new permutiples from known examples, and how previous efforts \cite{holt_3,holt_4} fit into a broader framework.

In subsequent sections, we make extensive use of several results and definitions from previous work \cite{holt_3,holt_5}. We now provide a summary of these.

\subsection{Basic definitions and results}

We use the notation $(d_k,d_{k-1},\ldots,d_0)_b$ to denote the natural number $\sum_{j=0}^{k}d_j b^j$, where $0\leq d_j<b$ for all $0\leq j \leq k$. With this notation, we state the definition of a permutiple number \cite[Definition 1]{holt_3}.

\begin{definition}
Let $n$ and $b$ be natural numbers where $1<n<b$, and let $\sigma$ be a permutation on $\{0,1,2,\ldots, k\}$. We say that $(d_k, d_{k-1},\ldots, d_0)_b$  is an $(n,b,\sigma)$-\textit{permutiple} provided
\[
(d_k,d_{k-1},\ldots,d_1, d_0)_b=n\cdot(d_{\sigma(k)},d_{\sigma(k-1)},\ldots, d_{\sigma(1)}, d_{\sigma(0)})_b.
\]
In the case that the digit permutation, $\sigma$, is not relevant, we may refer to $(d_k,d_{k-1},\ldots,d_0)_b$ as simply an $(n,b)$-permutiple. We refer to the collection of base-$b$ permutiples having $n$ as their  multiplier as {\it $(n,b)$-permutiples}.
\end{definition}

An algorithm for carrying out single-digit multiplication gives us the next result, which is stated and proved in previous work \cite[Theorem 1]{holt_3}.

\begin{theorem}
Let $(d_k, d_{k-1},\ldots, d_0)_b$ be an $(n,b,\sigma)$-permutiple, and let $c_j$ be the $j^{th}$ carry. Then, $b c_{j+1}-c_j=nd_{\sigma(j)}-d_{j}$ for all $0\leq j \leq k$.
\label{digits_carries_1}
\end{theorem}

Every carry in a digit-preserving multiplication is less than the multiplier \cite[Theorem 2]{holt_3}. This fact is of particular importance in what follows, and we state the full result below.

\begin{theorem} \label{carries}
Let $(d_k, d_{k-1},\ldots, d_0)_b$ be an $(n,b,\sigma)$-permutiple, and let $c_j$ be the $j^{th}$ carry. Then, $c_j\leq n-1$ for all $0 \leq j \leq k$.
\end{theorem}

For an $(n,b,\sigma)$-permutiple, $(d_k,d_{k-1},\ldots, d_0)_b$,  there are techniques for finding permutations, $\pi$, which yield another permutiple, $(d_{\pi(k)},d_{\pi(k-1)},\ldots, d_{\pi(0)})_b$. These methods rely on the notion of permutiple {\it conjugacy} to sort new examples into {\it conjugacy classes} \cite[Definition 2]{holt_3,holt_4}.

\begin{definition}\label{conj_class_def}
Let $(d_k, d_{k-1},\ldots, d_0)_b$ be an $(n,b)$-permutiple. Then, an $(n,b, \tau_1)$-permutiple, $(d_{\pi_1(k)}, d_{\pi_1(k-1)},\ldots, d_{\pi_1(0)})_b$, and an $(n,b, \tau_2)$-permutiple, $(d_{\pi_2(k)}, d_{\pi_2(k-1)},\ldots, d_{\pi_2(0)})_b$, are said to be \textit{conjugate} if $\pi_1 \tau_1 \pi_1^{-1}=\pi_2 \tau_2 \pi_2^{-1}$.
\end{definition}
In the case of repeated digits, this definition requires that we assume the collection $\{d_k, d_{k-1},\ldots, d_0\}$ is a multiset.

As an example of the above ideas, we may apply the techniques featured in previous work \cite{holt_3,holt_4} to a known example, $p=(d_4,d_3,d_2,d_1,d_0)_{10}=(8,7,9,1,2)_{10}=4 \cdot (2,1,9,7,8)_{10}$, to compute elements, $(d_{\pi(4)},d_{\pi(3)},d_{\pi(2)},d_{\pi(1)},d_{\pi(0)})_{10}$, of the conjugacy class containing $p$,  all of which are shown in Table \ref{conj_class_table}. Note that $\psi$ is the 5-cycle $(0,1,2,3,4)$, $\rho$ is the reversal permutation, and $\varepsilon$ is the identity permutation.
\begin{center}
\begin{tabular}{|c|c|c|}
\hline $(4,10,\tau)$-Example & $\pi$ & $\tau$ \\\hline
$(8,7,9,1,2)_{10}=4 \cdot (2,1,9,7,8)_{10}$ & $\varepsilon$ & $\rho$  \\\hline
$(8,7,1,9,2)_{10}=4 \cdot (2,1,7,9,8)_{10}$ & $(1,2)$ & $(1,2)\rho(1,2)$ \\\hline
$(7,9,1,2,8)_{10}=4 \cdot (1,9,7,8,2)_{10}$ & $\psi^4$ & $\psi^{-4}\rho\psi^4$ \\\hline
$(7,1,9,2,8)_{10}=4 \cdot (1,7,9,8,2)_{10}$ & $(1,2)\psi^4$ & $\psi^{-4}(1,2)\rho(1,2)\psi^4$ \\\hline
\end{tabular}
\captionof{table}{The conjugacy class of $p=(8,7,9,1,2)_{10}=4 \cdot (2,1,9,7,8)_{10}$.}\label{conj_class_table}
\end{center}
All of the examples in Table \ref{conj_class_table} are conjugate since $\pi\tau\pi^{-1}=\rho$ for every $\pi$ and $\tau$ listed. An example with the same digits which is not a member of the above conjugacy class is $(7,8,9,1,2)_{10} = 4 \cdot (1,9,7,2,8)_{10}$, which may also be found using previous work \cite{holt_4}.

\subsection{Permutiple graphs and the mother graph}

We now state a definition of a permutiple's graph \cite[Definition 3]{holt_5}, which is a description of the digit permutation in terms of the digits themselves rather than their indexing set.

\begin{definition}
Let $p=(d_k, d_{k-1},\ldots, d_0)_b$  be an $(n,b,\sigma)$-permutiple. We define a directed graph, called the {\it graph of $p$}, denoted by $G_p$, to consist of the collection of base-$b$ digits as vertices, and the collection of directed edges $E_p= \{(d_j,d_{\sigma(j)}) \mid  0 \leq j \leq k \}$.  A graph, $G$, for which there is a permutiple, $p$, such that $G=G_p$ is called a {\it permutiple graph}.
\label{class_graph}
\end{definition}

We note that for the remainder of this effort, all graphs are directed graphs, and we may refer to a ``directed graph'' as simply a ``graph,'' or a ``directed edge'' as an ``edge.''

A permutiple's conjugacy class gives us additional information about its graph; when two $(n,b)$-permutiples, $p_1$ and $p_2$, are conjugate, their graphs are the same \cite[Theorem 4]{holt_5}.

\begin{theorem}\label{conj_class}
Suppose $(d_k, d_{k-1},\ldots, d_0)_b$ is an $(n,b,\sigma)$-permutiple. Also, suppose $p_1=(d_{\pi_1(k)}, d_{\pi_1(k-1)},\ldots, d_{\pi_1(0)})_b$ is an $(n,b, \tau_1)$-permutiple, and $p_2=(d_{\pi_2(k)}, d_{\pi_2(k-1)},\ldots, d_{\pi_2(0)})_b$ is an $(n,b, \tau_2)$-permutiple. Then, if $p_1$ and $p_2$ are members of the same conjugacy class, then $G_{p_1}=G_{p_2}$.
\end{theorem}

Each conjugate permutiple in Table \ref{conj_class_table} has the same graph displayed in Figure \ref{conj_class_graph}.
\begin{center}
\begin{tikzpicture}
\tikzset{edge/.style = {->,> = latex'}}
\tikzset{vertex/.style = {shape=circle,draw,minimum size=1.5em}}
[xscale=3, yscale=3, auto=left,every node/.style={circle,fill=blue!20}]
\node[vertex] (n0) at (12,5) {$0$};
\node[vertex] (n1) at (11.618,6.17557) {$1$};
\node[vertex] (n2) at (10.618,6.90211) {$2$};
\node[vertex] (n3) at (9.38197,6.90211) {$3$};
\node[vertex] (n4) at (8.38197,6.17557) {$4$};
\node[vertex] (n5) at (8,5.00001) {$5$};
\node[vertex] (n6) at (8.38196,3.82443) {$6$};
\node[vertex] (n7) at (9.38196,3.09789) {$7$};
\node[vertex] (n8) at (10.618,3.09788) {$8$};
\node[vertex] (n9) at (11.618,3.82442) {$9$};
\draw[edge, bend right=10] (n8) to (n2);
\draw[edge, bend right=10] (n2) to (n8);
\draw[edge, bend right=10] (n7) to (n1);
\draw[edge, bend right=10] (n1) to (n7);
\draw[edge] (n9) to[in=0,out=-80, loop, style={min distance=10mm}] (n9);
\end{tikzpicture}
\captionof{figure}{The graph with edges $\{(d_j,d_{\tau(j)}) \mid 0 \leq j \leq 4 \}$ for each example in Table \ref{conj_class_table}.}
\label{conj_class_graph}
\end{center}

Permutiple graphs enable us to classify examples according to their graph \cite[Definition 6]{holt_5}. Although the example $p=(7,8,9,1,2)_{10} = 4 \cdot (1,9,7,2,8)_{10}$ has the same digits as those in Table \ref{conj_class_table}, the graph of $p$ immediately distinguishes it from the others as it belongs to an entirely different conjugacy class.

\begin{definition}\label{perm_class}
Let $p$ be an $(n,b)$-permutiple with graph $G_p$. We define the {\it class of $p$} to be the collection, $C$, of all $(n,b)$-permutiples, $q$, such that $G_q$ is a subgraph of $G_p$. We also define the graph of the class to be $G_p$, which we denote by $G_C$ and call the {\it graph of $C$.}
\end{definition}

Theorem \ref{conj_class}, coupled with the above definition, tells us that conjugate permutiples are always members of the same permutiple class.

The next result narrows down the possible collection of edges of a permutiple graph \cite[Theorem 3]{holt_5}.

\begin{theorem}
Let $p=(d_k, d_{k-1},\ldots, d_0)_b$ be an $(n,b,\sigma)$-permutiple with graph $G_p$. Then, for every edge, $(d_j,d_{\sigma(j)})$, of $G_p$, it must be that $\lambda \bigl(d_j+(b-n)d_{\sigma(j)}\bigr)\leq n-1$ for all $0\leq j\leq k$, where $\lambda$ gives the least non-negative residue modulo $b$.
\end{theorem}

The above allows us to gather all possible edges of a permutiple graph into a single graph \cite[Definition 4]{holt_5}.

\begin{definition}\label{mother_graph}\label{mg_def}
The $(n,b)$-\textit{mother graph}, denoted by $M$, is the graph having all base-$b$ digits as its vertices and the collection of all edges, $(d_1,d_2)$, which satisfy the inequality $\lambda\bigl(d_1+(b-n)d_{2}\bigr)\leq n-1$.
\end{definition}

The next result is also fundamental to the methods presented both here and in other work \cite[Theorem 6]{holt_5}.

\begin{theorem}
Let $C$ be an $(n,b)$-permutiple class. Then, $G_C$ is a union of cycles of $M$.
\label{cycle_union}
\end{theorem}

\begin{example}\label{example_00}
The $(4,10)$-mother graph is displayed in Figure \ref{4_10_mg}. Letting $p=(8,7,9,1,2)_{10}=4 \cdot (2,1,9,7,8)_{10}$ from Table \ref{conj_class_table}, and $C$ be the $(4,10)$-permutiple class with graph $G_C=G_p$, Figure \ref{4_10_mg} highlights the graph of $C$ in bold red.

\begin{center}
\begin{tikzpicture}
\tikzset{edge/.style = {->,> = latex'}}
\tikzset{vertex/.style = {shape=circle,draw,minimum size=1.5em}}
[xscale=2, yscale=2, auto=left,every node/.style={circle,fill=blue!20}]
\node[vertex] (n0) at (13,5) {$0$};
\node[vertex,red, very thick] (n1) at (12.427051918969068,6.763354468797628) {$\bf 1$};
\node[vertex,red, very thick] (n2) at (10.927054011580958,7.853168564878643) {$\bf 2$};
\node[vertex] (n3) at (9.07295355956129,7.853171024889661) {$3$};
\node[vertex] (n4) at (7.5729527602588975,6.76336090919162) {$4$};
\node[vertex] (n5) at (7.000000000010562,5.00000796076938) {$5$};
\node[vertex] (n6) at (7.572943401820056,3.236651971608781) {$6$};
\node[vertex,red, very thick] (n7) at (9.072938417283323,2.1468338951524664) {$\bf 7$};
\node[vertex,red, very thick] (n8) at (10.927038869289936,2.1468265151194124) {$\bf 8$};
\node[vertex,red, very thick] (n9) at (12.427042560496046,3.236632650426805) {$\bf 9$};
\draw[edge](n0) to[in=40,out=-40, loop, style={min distance=10mm}] (n0);
\draw[edge, bend right=5](n0) to (n2);
\draw[edge, bend right=0](n0) to (n5);
\draw[edge, bend right=5](n0) to (n7);
\draw[edge, bend right=0](n1) to (n0);
\draw[edge, bend right=0](n1) to (n2);
\draw[edge, bend right=5](n1) to (n5);
\draw[edge, red, bend right=5, very thick](n1) to (n7);
\draw[edge, bend right=5](n2) to (n0);
\draw[edge, bend right=0](n2) to (n3);
\draw[edge, bend right=0](n2) to (n5);
\draw[edge, red, bend right=5, very thick](n2) to (n8);
\draw[edge, bend right=0](n3) to (n0);
\draw[edge](n3) to[in=140,out=60, loop, style={min distance=10mm}] (n3);
\draw[edge, bend right=5](n3) to (n5);
\draw[edge, bend right=0](n3) to (n8);
\draw[edge, bend right=0](n4) to (n1);
\draw[edge, bend right=0](n4) to (n3);
\draw[edge, bend right=5](n4) to (n6);
\draw[edge, bend right=5](n4) to (n8);
\draw[edge, bend right=5](n5) to (n1);
\draw[edge, bend right=5](n5) to (n3);
\draw[edge, bend right=0](n5) to (n6);
\draw[edge, bend right=0](n5) to (n8);
\draw[edge, bend right=0](n6) to (n1);
\draw[edge, bend right=5](n6) to (n4);
\draw[edge](n6) to[in=190,out=270, loop, style={min distance=10mm}] (n6);
\draw[edge, bend right=0](n6) to (n9);
\draw[edge, red, bend right=5,very thick](n7) to (n1);
\draw[edge, bend right=0](n7) to (n4);
\draw[edge, bend right=0](n7) to (n6);
\draw[edge, bend right=5](n7) to (n9);
\draw[edge, red, bend right=5, very thick](n8) to (n2);
\draw[edge, bend right=5](n8) to (n4);
\draw[edge, bend right=0](n8) to (n7);
\draw[edge, bend right=0](n8) to (n9);
\draw[edge, bend right=0](n9) to (n2);
\draw[edge, bend right=0](n9) to (n4);
\draw[edge, bend right=5](n9) to (n7);
\draw[edge,red, very thick](n9) to[in=0,out=-80, loop, style={min distance=10mm}] (n9);
\end{tikzpicture}
\captionof{figure}{The $(4,10)$-mother graph with the graph of $C$ displayed in bold red.}
\label{4_10_mg}
\end{center}
\end{example}

\subsection{Finite-state-machine description of general digit-preserving multiplication and the Hoey-Sloane graph}

To create a framework for describing digit-preserving multiplication, we construct a finite-state machine which describes single-digit multiplication by a number, $n$, less than the base, $b$. Here, the carries take on a central role in the discussion. By Theorem \ref{carries}, all of the carries in a digit-preserving multiplication are less than the multiplier, $n$. Thus, the possible states of the machine are non-negative integers less than $n$. The input alphabet is the collection of edges of the mother graph, $M$, and the statement of Theorem \ref{digits_carries_1} motivates the definition of the state-transition function,
\begin{equation}\label{state_transition}
 c_2=(nd_{2}-d_1+c_1)\div b,
\end{equation}
where the input $(d_1,d_2)$ enables a transition from state $c_1$ to state $c_2$. This transition defines a labeled edge on the state diagram as seen in Figure \ref{stae_diagram}.
\begin{center}
\begin{tikzpicture}
\tikzset{edge/.style = {->,> = latex'}}
\tikzset{vertex/.style = {shape=circle,draw,minimum size=1.5em}}
[xscale=2, yscale=2, auto=left,every node/.style={circle,fill=blue!20}]
\node[vertex] (n0) at (7,5) {$c_1$};
\node[vertex] (n1) at (13,5) {$c_2$};
\draw[edge] (n0) edge node[above] {$(d_1,d_2)$} (n1);
\end{tikzpicture}
\captionof{figure}{A labeled edge on the state diagram.}
\label{stae_diagram}
\end{center}
The first carry, $c_0$, of a single-digit multiplication is zero by definition \cite{holt_3}. In the present context, this is to say that the only possible initial state is zero. Also, assuming the product of a single-digit multiplication is a $(k+1)$-digit number, $(d_k,d_{k-1},\ldots,d_0)$, the state $c_{k+1}$ must also be zero, otherwise, the result would be a $(k+2)$-digit number. Thus, the zero state is the only possible accepting state. We call the above machine the $(n,b)$-{\it Hoey-Sloane machine}, and its state diagram is called the $(n,b)$-{\it Hoey-Sloane graph}, which we denote by $\Gamma$.

We note here that an input, $(d_1,d_2)$, which allows a transition from state $c_1$ to state $c_2$ is generally not unique. In this way, the edge label on the state diagram from $c_1$ to $c_2$ may not be a single input, but a list of inputs. That is, we assign a collection of suitable inputs to each edge on $\Gamma$ by the mapping  $(c_1,c_2) \mapsto \{(d_1,d_2)\in E_M \mid c_2=(nd_{2}-d_1+c_1)\div b\}$, where $E_M$ is the collection of edges of $M$. We also note that we could alternatively define a multigraph where a unique multi-edge is assigned to each input in $E_M$, as done in other work \cite{holt_6}. We give some details of this construction later.

We let $L$ denote the  language of input strings accepted by the $(n,b)$-Hoey-Sloane machine. We may describe $L$ as finite sequences of edge-label inputs which define walks on $\Gamma$ whose initial and final states are the zero state. Such walks we call {\it $L$-walks}. Members of $L$ which produce permutiple numbers are called {\it $(n,b)$-permutiple strings.}

In this new setting, Theorem \ref{cycle_union} gives us the following result \cite[Corollary 1]{holt_5}.

\begin{corollary}\label{cycle_union_2}
Let $s=(d_{0},\widehat{d_{0}})(d_{1},\widehat{d_{1}})\cdots(d_{k},\widehat{d_{k}})$ be a member of $L$. If $s$ is a permutiple string, then the collection of ordered-pair inputs of $s$ is a union of cycles of $M$.
\end{corollary}

The converse of Corollary \ref{cycle_union_2} is not true in general; it is easy to find members of $L$ whose ordered pairs make up cycles on $M$, but are not permutiple strings \cite{holt_5}. In order for an input string, $s=(d_{0},\widehat{d_{0}})(d_{1},\widehat{d_{1}})\cdots(d_{k},\widehat{d_{k}})$, in $L$ to qualify as a permutiple string, it must be that the two collections of base-$b$ digits formed by the left and right components of the inputs of $s$ must be the same. In other words, $\{d_{k},\ldots, d_{1}, d_{0}\}$ and $\{\widehat{d_{k}},\ldots \widehat{d_{1}}, \widehat{d_{0}}\}$ must form the same multiset. Since every member of $L$ describes a valid single-digit, base-$b$ multiplication by $n$, multiset unions of cycles of $M$ which can be ordered into an element, $s$, of $L$, allow us to form permutiple strings. Since $s$ may be visualized as an $L$-walk on  the $(n,b)$-Hoey-Sloane graph, $\Gamma$,  the edges of $\Gamma$ associated with the inputs of $s$ must define a strongly-connected subgraph of $\Gamma$ containing the zero state.

It is now advantageous to reestablish some additional terminology and notation used in previous efforts \cite{holt_5}. In this article, a union of multisets is denoted by $\uplus$. For instance, $\{1,2,2,3\} \uplus \{2,3,4\} = \{1,2,2,2,3,3,4\}$. Let $C_0,C_1,\ldots,C_m$ be the cycles of $M$. For each $C_j$, construct a labeled subgraph of $\Gamma$ which we denote by $\Gamma_j$.  The edges of $\Gamma_j$ are pairs, $(c_1,c_2)$, for which the collection $\mathscr{E}_j=\{(d_1,d_2)\in C_j \mid c_2=(nd_{2}-d_1+c_1)\div b\}$ is nonempty. The edge label is then the list of elements of $\mathscr{E}_j$. With the collection $\mathscr{E}_j$ in hand, we exclude any state from the collection of vertices where the indegree and outdegree are both zero. We refer to $\Gamma_j$ as the {\it image of $C_j$,} or simply as a {\it cycle image}. We now suppose $I$ is a multiset whose support is a subset, $J$, of $\{0,1,\ldots,m\}$. The key observation in previous efforts \cite{holt_5} is that $L$-walks can only occur on subgraphs of $\Gamma$ which are strongly connected and contain the zero state. Thus, our search for permutiple strings is reduced to considering cycle image unions, $\Gamma_{J}=\bigcup_{j\in J}\Gamma_j$ (edge labels included), which are strongly-connected  and contain the zero state. If the corresponding multiset union of mother-graph cycles, $C_{I}=\biguplus_{j\in I}C_j$, can be ordered into a string, $s$, belonging to $L$, then $s$ is a permutiple string. This is to say that a strongly-connected graph, $\Gamma_J$, containing the zero state is a necessary condition for being able to order the elements of $C_I$ into a permutiple string. Later, in Example \ref{example_000}, we provide a counterexample which shows that these conditions are not sufficient for this same purpose.

\begin{example}\label{example_0}
The cycles of $G_C$ from Example \ref{example_00}  are $C_0=\{(9,9)\}$, $C_1=\{(2,8),(8,2)\}$, and $C_2=\{(1,7),(7,1)\}$. Table \ref{4_10_mi} depicts the cycle images corresponding to the cycles of $G_C$.

\begin{center}
\begin{tabular}{|c|c|l|c|}
\hline
&
{\bf Cycle of $G_C$}
& {\bf Cycle Image}
&\\\hline
$C_0$
&
\begin{tikzpicture} \tikzset{edge/.style = {->,> = latex'}} \tikzset{vertex/.style = {shape=circle,draw,minimum size=1.5em}} [xscale=2, yscale=2, auto=left,every node/.style={circle,fill=blue!20}]
\node[vertex] (n1) at (0,0) {$9$};
\draw[edge](n1) to[in=0,out=-80, loop, style={min distance=7mm}] (n1);
\end{tikzpicture}
&
\begin{tikzpicture}
\tikzset{edge/.style = {->,> = latex'}}
\tikzset{vertex/.style = {shape=circle,draw,minimum size=1.5em}}
[xscale=2, yscale=2, auto=left,every node/.style={circle,fill=blue!20}]
\color{white}\node[vertex,accepting,initial,white] (n0) at (0,5) {$0$};
\color{black}\node[vertex] (n1) at (3,5) {$3$};
\color{black}\draw[edge](n1) to[in=90,out=30, loop, style={min distance=10mm}] node[pos=0.25,right] {$(9,9)$} (n1);
\end{tikzpicture}
&
$\Gamma_0$
\\\hline
$C_1$
&
\begin{tikzpicture}
\tikzset{edge/.style = {->,> = latex'}}
\tikzset{vertex/.style = {shape=circle,draw,minimum size=1.5em}}
[xscale=3, yscale=3, auto=left,every node/.style={circle,fill=blue!20}]
\node[vertex] (n2) at (0,2) {$2$};
\node[vertex] (n8) at (0,0) {$8$};
\draw[edge, bend right=10] (n8) to (n2);
\draw[edge, bend right=10] (n2) to (n8);
\end{tikzpicture}
&
\begin{tikzpicture}
\tikzset{edge/.style = {->,> = latex'}}
\tikzset{vertex/.style = {shape=circle,draw,minimum size=1.5em}}
[xscale=2, yscale=2, auto=left,every node/.style={circle,fill=blue!20}]
\node[vertex,accepting,initial] (n0) at (0,5) {$0$};
\node[vertex] (n1) at (3,5) {$3$};
\draw[edge,bend right=10] (n0) edge node[below] {$(2,8)$} (n1);
\draw[edge](n0) to[in=150,out=90, loop, style={min distance=10mm}] node[pos=0.75,left] {$(8,2)$} (n0);
\end{tikzpicture}
&
$\Gamma_1$
\\\hline
$C_2$
&
\begin{tikzpicture}
\tikzset{edge/.style = {->,> = latex'}}
\tikzset{vertex/.style = {shape=circle,draw,minimum size=1.5em}}
[xscale=3, yscale=3, auto=left,every node/.style={circle,fill=blue!20}]
\node[vertex] (n1) at (1.3,1.7) {$1$};
\node[vertex] (n7) at (0,0) {$7$};
\draw[edge, bend right=10] (n7) to (n1);
\draw[edge, bend right=10] (n1) to (n7);
\end{tikzpicture}
&
\begin{tikzpicture}
\tikzset{edge/.style = {->,> = latex'}}
\tikzset{vertex/.style = {shape=circle,draw,minimum size=1.5em}}
[xscale=2, yscale=2, auto=left,every node/.style={circle,fill=blue!20}]
\node[vertex,accepting,initial] (n0) at (0,5) {$0$};
\node[vertex] (n1) at (3,5) {$3$};
\draw[edge, bend right=10] (n1) edge node[above] {$(7,1)$} (n0);
\draw[edge](n1) to[in=90,out=30, loop, style={min distance=10mm}] node[pos=0.25,right] {$(1,7)$} (n1);
\end{tikzpicture}
&
$\Gamma_2$
\\\hline
\end{tabular}
\captionof{table}{Cycles of $G_C$ and their corresponding cycle images.}\label{4_10_mi}
\end{center}

The $(4,10)$-Hoey-Sloane graph is shown in Figure \ref{4_10_hsg}. The graph of the union of the cycle images, $\Gamma_J=\Gamma_0 \cup \Gamma_1 \cup \Gamma_2$, corresponding to the cycles of $G_C$ is highlighted in bold red.

\begin{center}
\begin{tikzpicture}
\tikzset{edge/.style = {->,> = latex'}}
\tikzset{vertex/.style = {shape=circle,draw,minimum size=1.5em}}
[xscale=2, yscale=2, auto=left,every node/.style={circle,fill=blue!20}]
\node[vertex,red,initial,accepting,very thick] (n0) at (0,5) {$\bf 0$};
\node[vertex] (n1) at (3.0,5) {$1$};
\node[vertex] (n2) at (6.0,5) {$2$};
\node[vertex,red,very thick] (n3) at (9.0,5) {$\bf 3$};
\draw[edge,red,very thick](n0) to[in=150,out=90, loop, style={min distance=10mm}] node[left] {\scriptsize{$\color{black}{(0,0), (4,1), }\,\color{red}{\bf (8,2)}$}} (n0);
\draw[edge](n1) to[in=110,out=70, loop, style={min distance=7mm}] node[above] {\scriptsize{$(3,3),(7,4)$}} (n1);
\draw[edge](n2) to[in=110,out=70, loop, style={min distance=7mm}] node[above] {\scriptsize{$(2,5),(6,6)$}} (n2);
\draw[edge,red,very thick](n3) to[in=90,out=30, loop, style={min distance=10mm}] node[right] {\scriptsize{$\color{red}{\bf (1,7)}\color{black}{,(5,8),}\color{red}{\bf (9,9)}$}} (n3);
\draw[edge, bend left=10](n0) edge node[above] {\scriptsize{$(2,3),(6,4)$}} (n1);
\draw[edge, bend left=10](n1) edge node[below] {\scriptsize{$(1,0),(5,1),(9,2)$}} (n0);
\draw[edge, bend left=10](n2) edge node[below] {\scriptsize{$(0,2),(4,3),(8,4)$}} (n1);
\draw[edge, bend left=10](n1) edge node[above] {\scriptsize{$(1,5),(5,6),(9,7)$}} (n2);
\draw[edge, bend left=55](n1) edge node[above] {\scriptsize{$(3,8),(7,9)$}} (n3);
\draw[edge, bend right=10](n2) edge node[below] {\scriptsize{$(0,7),(4,8),(8,9)$}} (n3);
\draw[edge, bend right=10](n3) edge node[above] {\scriptsize{$(3,5),(7,6)$}} (n2);
\draw[edge, bend left=50](n3) edge node[below] {\scriptsize{$(1,2),(5,3),(9,4)$}} (n1);
\draw[edge, bend right=50](n0) edge node[below] {\scriptsize{$(0,5),(4,6),(8,7)$}} (n2);
\draw[edge, red, bend right=65,very thick](n0) edge node[below] {\scriptsize{${\bf (2,8)}\color{black}{,(6,9)}$}} (n3);
\draw[edge, bend right=55](n2) edge node[above] {\scriptsize{$(2,0),(6,1)$}} (n0);
\draw[edge, red, bend right=65,very thick](n3) edge node[above] {\scriptsize{$\color{black}{(3,0),}\,\color{red}{{\bf (7,1)}}$}} (n0);
\end{tikzpicture}
\captionof{figure}{The $(4,10)$-Hoey-Sloane graph with cycle images of $G_C$ in bold red.}
\label{4_10_hsg}
\end{center}

As a concrete example of the above ideas, consider the multiset union of mother-graph cycles
\[
C_I=C_0 \uplus C_1\uplus C_1 \uplus C_2 \uplus C_2
=\{(9,9),(8,2),(8,2),(2,8),(2,8),(7,1),(7,1),(1,7),(1,7)\}.
\]
With the help of $\Gamma_J$, there are multiple ways to order $C_I$ into a member of $L$. For instance, $(8,2)(8,2)(2,8)(9,9)(1,7)(1,7)(7,1)(2,8)(7,1)$ is an input string in $L$ which gives the $(4,10)$-permutiple $(7,2,7,1,1,9,2,8,8)_{10}=4\cdot(1,8,1,7,7,9,8,2,2)_{10}$.
Another example, $(8,2)(2,8)(1,7)(7,1)(2,8)(9,9)(1,7)(7,1)(8,2)$, corresponds to $(8,7,1,9,2,7,1,2,8)_{10}=4\cdot(2,1,7,9,8,1,7,8,2)_{10}$.
\end{example}

\subsection{A multigraph representation of the Hoey-Sloane machine}
The question remaining from the above is which multiset unions of mother-graph cycles yield permutiple strings. The author used a multigraph construction to answer this question \cite{holt_6}. Instead of assigning a collection of mother-graph inputs, $\mathscr{E}_j$, to a single edge, $(c_1,c_2)$, as described above, we may define a multiset of labeled multi-edges where each element of $\mathscr{E}_j$ is individually assigned an edge, $(c_1,c_2)$. That is, distinct inputs, $(d_1,d_2)$ and $(\widehat{d}_1,\widehat{d}_2)$, determine two distinct multi-edges connecting state $c_1$ to state $c_2$, as depicted in Figure \ref{stae_diagram_2}.
\begin{center}
\begin{tikzpicture}
\tikzset{edge/.style = {->,> = latex'}}
\tikzset{vertex/.style = {shape=circle,draw,minimum size=1.5em}}
[xscale=2, yscale=2, auto=left,every node/.style={circle,fill=blue!20}]
\node[vertex] (n0) at (7,5) {$c_1$};
\node[vertex] (n1) at (13,5) {$c_2$};
\draw[edge, bend left=10] (n0) edge node[above] {$(d_1,d_2)$} (n1);
\draw[edge, bend left=40] (n0) edge node[above] {$(\widehat{d}_1,\widehat{d}_2)$} (n1);
\end{tikzpicture}
\captionof{figure}{Two multi-edges representing two possible transitions from state $c_1$ to state $c_2$.}
\label{stae_diagram_2}
\end{center}
From previous work, we know that an edge label, $(d_1,d_2)$, cannot appear on distinct edges of $\Gamma$ \cite[Theorem 5]{holt_6}, and that every input, $(d_1,d_2)$, enables some transition, $(c_1,c_2)$ \cite[Theorem 6]{holt_6}. Thus, for each input edge, $(d_1,d_2)$, of $M$, we may find a unique pair, $(c_1,c_2)$, which solves Equation (\ref{state_transition}). That is, letting $N$ be the collection of non-negative integers less than the multiplier, $n$, the above determines a well-defined mapping, $\mu: E_M \rightarrow N \times N \times E_M$, given by $(d_1,d_2) \mapsto (c_1,c_2,(d_1,d_2))$. From there, we construct the {\it multi-image of $C_j$}, denoted by $\Delta_j$, in similar fashion to the cycle images. To define the edges of $\Delta_j$, apply $\mu$ to each element of the mother-graph cycle $C_j$ to obtain its corresponding labeled edge, $(c_1,c_2,(d_1,d_2))$. As with the cycle images, we exclude any state from the collection of vertices where the indegree and outdegree are both zero. We may now state conditions which are both necessary and sufficient for the existence of permutiple strings \cite[Corollary 2]{holt_6}.

\begin{corollary}\label{indegree_outdegree}
Let $\{C_0,C_1,\ldots,C_m\}$ be the collection of cycles of $M$, and let $\Delta_j$ be the corresponding multi-image of $C_j$, Also, let $I$ be a multiset whose support is a subset of $\{0,1,\ldots,m\}$, Then, a multiset union of mother-graph cycles, $C_I=\biguplus_{j\in I}C_j$, may be ordered into a permutiple string if and only if the corresponding multigraph union of the cycle multi-images, $\Delta_I=\biguplus_{j\in I}\Delta_j$, contains the zero state, is strongly connected, and the indegree is equal to the outdegree at each vertex.
\end{corollary}

If we apply $\mu$ to all the edges of $M$, we obtain the {\it $(n,b)$-Hoey-Sloane multigraph} \cite{holt_6}, which we denote by $\Delta$ to distinguish it from $\Gamma$. Although $\Gamma$ and $\Delta$ are different objects, they are both equivalent representations of the Hoey-Sloane machine, as mentioned earlier.

\begin{example}\label{example_000}
Picking up where Example \ref{example_0} leaves off, we note that the multigraph representation of each cycle multi-image, $\Delta_j$, looks the same as the graph of its cycle-image counterpart, $\Gamma_j$, shown in Table \ref{4_10_mi}. The multi-image union, $\Delta_I=\Delta_0 \uplus \Delta_1 \uplus \Delta_1 \uplus \Delta_2 \uplus \Delta_2$, corresponding to $C_I$ from Example \ref{example_0}, is shown in Figure \ref{4_10_mi_union_1}.  Since $\Delta_I$ contains the zero state, is strongly connected, and the indegree is equal to the outdegree at each vertex, Corollary \ref{indegree_outdegree} guarantees that we may find orderings like the one given in Example \ref{example_0}. In fact, we may use Figure \ref{4_10_mi_union_1} to determine all orderings by traversing all Eulerian circuits beginning and ending with the zero state.
\begin{center}
\begin{tikzpicture}
\tikzset{edge/.style = {->,> = latex'}}
\tikzset{vertex/.style = {shape=circle,draw,minimum size=1.5em}}
[xscale=2, yscale=2, auto=left,every node/.style={circle,fill=blue!20}]
\node[vertex,accepting,initial] (n0) at (0,0) {$0$};
\node[vertex] (n1) at (4,0) {$3$};
\draw[edge,bend right=50] (n0) edge node[below] {$(2,8)$} (n1);
\draw[edge,bend right=10] (n0) edge node[below] {$(2,8)$} (n1);
\draw[edge, bend right=10] (n1) edge node[above] {$(7,1)$} (n0);
\draw[edge, bend right=50] (n1) edge node[above] {$(7,1)$} (n0);
\draw[edge](n0) to[in=140,out=80, loop, style={min distance=7mm}] node[above] {$(8,2)$} (n0);
\draw[edge](n0) to[in=160,out=60, loop, style={min distance=30mm}] node[above] {$(8,2)$} (n0);
\draw[edge](n1) to[in=100,out=40, loop, style={min distance=7mm}] node[above] {$(1,7)$} (n1);
\draw[edge](n1) to[in=120,out=20, loop, style={min distance=30mm}] node[above] {$(1,7)$} (n1);
\draw[edge](n1) to[in=-40,out=-100, loop, style={min distance=7mm}] node[below] {$(9,9)$} (n1);
\end{tikzpicture}
\captionof{figure}{The multi-image union $\Delta_I=\Delta_0 \uplus \Delta_1 \uplus \Delta_1 \uplus \Delta_2 \uplus \Delta_2$ corresponding to the multiset union $C_I=C_0 \uplus C_1\uplus C_1 \uplus C_2 \uplus C_2$ of mother-graph cycles.}
\label{4_10_mi_union_1}
\end{center}

We also mentioned in Example \ref{example_0} that not every multiset union of mother-graph cycles, $C_I$, can be ordered into a permutiple string. As an example, consider $C_I=C_1\uplus C_1 \uplus C_2$. Inspecting the corresponding multi-image union, $\Delta_I=\Delta_1 \uplus \Delta_1 \uplus \Delta_2$, which is depicted in Figure \ref{4_10_mi_union_2}, we see that ordering $C_I$ into a member of $L$ is not possible. We may arrive at this same conclusion by using Corollary \ref{indegree_outdegree}; the indegrees and outdegrees are not equal for both vertices.

\begin{center}
\begin{tikzpicture}
\tikzset{edge/.style = {->,> = latex'}}
\tikzset{vertex/.style = {shape=circle,draw,minimum size=1.5em}}
[xscale=2, yscale=2, auto=left,every node/.style={circle,fill=blue!20}]
\node[vertex,accepting,initial] (n0) at (0,0) {$0$};
\node[vertex] (n1) at (4,0) {$3$};
\draw[edge,bend right=50] (n0) edge node[below] {$(2,8)$} (n1);
\draw[edge,bend right=10] (n0) edge node[below] {$(2,8)$} (n1);
\draw[edge, bend right=10] (n1) edge node[above] {$(7,1)$} (n0);
\draw[edge](n0) to[in=140,out=80, loop, style={min distance=7mm}] node[above] {$(8,2)$} (n0);
\draw[edge](n0) to[in=160,out=60, loop, style={min distance=30mm}] node[above] {$(8,2)$} (n0);
\draw[edge](n1) to[in=100,out=40, loop, style={min distance=7mm}] node[above] {$(1,7)$} (n1);
\end{tikzpicture}
\captionof{figure}{The multi-image union $\Delta_I=\Delta_1 \uplus \Delta_1 \uplus \Delta_2$ corresponding to the multiset union $C_I=C_1\uplus C_1 \uplus C_2$ of mother-graph cycles.}
\label{4_10_mi_union_2}
\end{center}
\end{example}

\section{Reflective symmetry of the mother graph and Hoey-Sloane graph}

We begin with a definition.

\begin{definition}\label{graph_sym}
Let $G$ be a directed graph whose vertices are base-$b$ digits, $B=\{0,1,\ldots,b-1\}$, with a collection of directed edges, $E$. Also, let $\overline{d}_1=b-1-d_1$ and $\overline{d}_2=b-1-d_2$, where $(d_1,d_2)$ is an edge of $G$. Finally, let $\overline{E}$ be the collection of directed edges $\{(\overline{d}_1,\overline{d}_2) \mid (d_1,d_2)\in E\}$. Then, the directed graph having $B$ as vertices and $\overline{E}$ as edges is called the {\it reflection} of $G$ and is denoted by $\overline{G}$. If $G$ is its own reflection, that is, if $G=\overline{G}$, then we say that $G$ is {\it symmetric}.
\end{definition}

\begin{remark}
The reader will notice that to obtain the reflection of a graph, we simply apply the reversal permutation, $\rho$, to the vertices. Then, symmetric graphs are nothing more than those for which $\rho$ is a graph automorphism.
\end{remark}

The following is a basic consequence of the above definition.

\begin{corollary}\label{un_of_refl_is_refl_of un_1}
Let $G_1$ and $G_2$ be subgraphs of a directed graph whose vertices are base-$b$ digits. Then, $\overline{G_1\cup G_2}=\overline{G}_1 \cup \overline{G}_2$.
\end{corollary}

Supposing that $(d_1,d_2)$ is an edge of $M$, we notice that
\[
\overline{d}_1 + (b-n)\overline{d}_2
\equiv n-1-(d_1+(b-n)d_2)\pmod{b}.
\]
Since $\lambda(d_1+(b-n)d_2)\leq n-1$ by definition, we have $\lambda(\overline{d}_1 + (b-n)\overline{d}_2)\leq n-1$. This observation yields our first new theorem.

\begin{theorem}
The $(n,b)$-mother graph is symmetric.
\end{theorem}

From the above, we may say that reflections of mother-graph cycles are again mother-graph cycles. We also draw the reader's attention to the reflective symmetry of the $(4,10)$-mother graph in Figure \ref{4_10_mg}.

Now, for an edge, $(c_1,c_2)$, on $\Gamma$ with edge label $(d_1,d_2)$, we notice that if we let $\overline{c}_1=n-1-c_1$ and $\overline{c}_2=n-1-c_2$ (treating $c_1$ and $c_2$ as base-$n$ digits), then
\begin{align*}
(n\overline{d}_2-\overline{d}_1+\overline{c}_1)\div b
&=\big(n(b-1-d_2)-(b-1-d_1)+n-1-c_1\big)\div b\\
&=n-1-c_2\\
&=\overline{c}_2.
\end{align*}
Since $c_2\leq n-1$, we have $\overline{c}_2 \leq n-1$. These facts give us our next result.

\begin{theorem}\label{mg_hs_reflection}
If $(c_1,c_2)$ is an edge on the $(n,b)$-Hoey-Sloane graph with edge label $(d_1,d_2)$, then $(\overline{c}_1,\overline{c}_2)$ is also an edge on the $(n,b)$-Hoey-Sloane graph with edge label $(\overline{d}_1,\overline{d}_2)$.
\end{theorem}

Theorem \ref{mg_hs_reflection} tells us, in a sense to be made precise below, that $\Gamma$ inherits the reflective symmetry of $M$.
\begin{center}
\begin{tabular}{|c|c|}
\hline
{\bf Edge on Mother Graph} & {\bf Labeled Edge on Hoey-Sloane Graph}\\\hline
\begin{tikzpicture}
\tikzset{edge/.style = {->,> = latex'}}
\tikzset{vertex/.style = {shape=circle,draw,minimum size=1.5em}}
[xscale=2, yscale=2, auto=left,every node/.style={circle,fill=blue!20}]
\node[vertex] (n0) at (7,5) {$d_1$};
\node[vertex] (n1) at (10,5) {$d_2$};
\draw[edge] (n0) edge (n1);
\end{tikzpicture}
&
\begin{tikzpicture}
\tikzset{edge/.style = {->,> = latex'}}
\tikzset{vertex/.style = {shape=circle,draw,minimum size=1.5em}}
[xscale=2, yscale=2, auto=left,every node/.style={circle,fill=blue!20}]
\node[vertex] (n0) at (7,5) {$c_1$};
\node[vertex] (n1) at (10,5) {$c_2$};
\draw[edge] (n0) edge node[above] {$(d_1,d_2)$} (n1);
\end{tikzpicture}
\\\hline
\begin{tikzpicture}
\tikzset{edge/.style = {->,> = latex'}}
\tikzset{vertex/.style = {shape=circle,draw,minimum size=1.5em}}
[xscale=2, yscale=2, auto=left,every node/.style={circle,fill=blue!20}]
\node[vertex] (n0) at (7,5) {$\overline{d}_1$};
\node[vertex] (n1) at (10,5) {$\overline{d}_2$};
\draw[edge] (n0) edge (n1);
\end{tikzpicture}
&
\begin{tikzpicture}
\tikzset{edge/.style = {->,> = latex'}}
\tikzset{vertex/.style = {shape=circle,draw,minimum size=1.5em}}
[xscale=2, yscale=2, auto=left,every node/.style={circle,fill=blue!20}]
\node[vertex] (n0) at (7,5) {$\overline{c}_1$};
\node[vertex] (n1) at (10,5) {$\overline{c}_2$};
\draw[edge] (n0) edge node[above] {$(\overline{d}_1,\overline{d}_2)$} (n1);
\end{tikzpicture}
\\\hline
\end{tabular}
\captionof{table}{Reflective symmetry of the $(n,b)$-mother graph and the $(n,b)$-Hoey-Sloane graph.}
\label{symmetry}
\end{center}

From the above, we may say that reflections of cycle images are cycle images of reflections of mother-graph cycles. We state the above more precisely as a corollary, but first we need to take care of some details since the edge labels of the Hoey-Sloane graph make it a bit more complex than the mother graph. In particular, we define what we mean by the reflection of a subgraph of $\Gamma$.

\begin{definition}\label{hs_subgraph_reflection}
Let $\Gamma_0$ be a subgraph of $\Gamma$. We define the {\it reflection} of $\Gamma_0$, denoted by $\overline{\Gamma}_0$, to be the edge-labeled graph with the following properties:
\begin{enumerate}
\item $\overline{c}$ is a vertex of $\overline{\Gamma}_0$ whenever $c$ is a vertex of $\Gamma_0$.
\item $(\overline{c}_1,\overline{c}_2)$ is an edge of $\overline{\Gamma}_0$ whenever $(c_1,c_2)$ is an edge of $\Gamma_0$.
\item $(\overline{d}_1,\overline{d}_2)$ is an edge label of $\overline{\Gamma}_0$ whenever $(d_1,d_2)$ is an edge label of $\Gamma_0$.
\end{enumerate}
If $\Gamma_0$ is its own reflection, that is, if $\Gamma_0=\overline{\Gamma}_0$, then we say $\Gamma_0$ is {\it symmetric}.
\end{definition}

With the above, we may say more precisely what we mean by the ``reflective symmetry'' of $\Gamma$. In particular, with Definition \ref{hs_subgraph_reflection}, Theorem \ref{mg_hs_reflection} gives us the following corollaries.

\begin{corollary}
The $(n,b)$-Hoey-Sloane graph is symmetric.
\end{corollary}

The reader may observe the reflective symmetry of the $(4,10)$-Hoey-Sloane graph in Figure \ref{4_10_hsg}.

\begin{corollary}\label{cycle_image_reflection}
Let $C_0$ be a cycle of $M$. Then, $\Gamma_0$ is the cycle image of $C_0$ if and only if $\overline{\Gamma}_0$ is the cycle image of  $\overline{C}_0$.
\end{corollary}

$\overline{\Gamma}_0$ also inherits the topology of $\Gamma_0$, which gives us the following.

\begin{corollary}\label{hs_topology}
Let $\Gamma_0$ be a strongly-connected subgraph of $\Gamma$. Then, $\overline{\Gamma}_0$ is a strongly-connected subgraph of $\Gamma$.
\end{corollary}

The following defines some useful notation.

\begin{definition}
Suppose that $\mathscr{C}=\{C_0,\ldots,C_m\}$ is the collection of cycles of $G_C$, and suppose that $\{\Gamma_0,\ldots,\Gamma_m\}$ is the collection of corresponding cycle images of $\mathscr{C}$. We denote the union of cycle images, $\bigcup_{j=0}^{m}\Gamma_j$, by $\Gamma_C$.
\end{definition}

Letting $C$ be a permutiple class, we may now say when the reflection, $\overline{G}_C$, of $G_C$, is also the graph of a permutiple class.

\begin{theorem}\label{equiv_reflection}
Let $C$ be an $(n,b)$-permutiple class, and let $G_C$ be its graph. Then, the following statements are equivalent.
\begin{enumerate}
\item $\overline{G}_C$ is the graph of a permutiple class.
\item The zero state is a vertex of $\overline{\Gamma}_C$.
\item The state $n-1$ is a vertex of $\Gamma_C$.
\end{enumerate}
\end{theorem}

\begin{proof}
$({\it 1} \implies {\it 2})$.
If $\overline{G}_C$ is the graph of a permutiple class, then the zero state is trivially a vertex of $\overline{\Gamma}_C$.

$({\it 2} \implies {\it 3})$. Now suppose that the zero state is a vertex of $\overline{\Gamma}_C$. Then, by Theorem \ref{mg_hs_reflection} and Definition \ref{hs_subgraph_reflection}, the state $\overline{0}=n-1$ must be a vertex of $\Gamma_C$.

$({\it 3} \implies {\it 1})$. Suppose that the state $c_j=n-1$ is a vertex of $\Gamma_C$. Since $G_C$ is the graph of the permutiple class $C$, we may choose an $(n,b,\sigma)$-permutiple, $p=(d_k,\ldots,d_0)_b=n\cdot (d_{\sigma(k)},\ldots,d_{\sigma(0)})_b$, with carries $c_k,\ldots, c_j=n-1,\ldots, c_1$, $c_0=0$, such that $G_C=G_p$. Let $s=(d_{0},d_{\sigma(0)})\cdots (d_{k},d_{\sigma(k)})$ be the permutiple string of $p$, and let
\[
S=\{(0,c_{1}),(c_{1},c_{2}),\ldots,(c_{j-1},n-1),(n-1,c_{j+1})\ldots,(c_{k-1},c_{k}),(c_{k},0)\}
\] be the associated sequence of state transitions induced by the inputs of $s$ on the $L$-walk of $p$, which is visualized in Figure \ref{generic_L_walk}.
\begin{center}
\begin{footnotesize}
\begin{tikzpicture}
\tikzset{edge/.style = {->,> = latex'}}
\tikzset{vertex/.style = {shape=circle,draw,minimum size=1.0em}}
[xscale=2, yscale=2, auto=left,every node/.style={circle,fill=blue!20}]
\node[vertex] (n0) at (0,2.5) {$0$};
\node[vertex] (n1) at (2,2.5) {$c_1$};
\node[vertex] (n2) at (4,2.5) {$c_2$};
\draw[edge] (n0) edge node[above] {$(d_0,d_{\sigma(0)})$} (n1);
\draw[edge] (n1) edge node[above] {$(d_1,d_{\sigma(1)})$} (n2);
\node (2) at (4.65,2.5) {$\ldots$};
\node[vertex] (n0) at (5.5,2.5) {$c_{j-1}$};
\node[vertex] (n1) at (9,2.5) {$n-1$};
\node[vertex] (n2) at (11.5,2.5) {$c_{j+1}$};
\draw[edge] (n0) edge node[above] {$(d_{j-1},d_{\sigma(j-1)})$} (n1);
\draw[edge] (n1) edge node[above] {$(d_j,d_{\sigma(j)})$} (n2);
\node (3) at (12.3,2.5) {$\ldots$};
\path (3) -- node[auto=false]{} (2);
\node[vertex] (n0) at (13,2.5) {$c_k$};
\node[vertex] (n1) at (15,2.5) {$0$};
\draw[edge] (n0) edge node[above] {$(d_k,d_{\sigma(k)})$} (n1);
\end{tikzpicture}
\end{footnotesize}
\captionof{figure}{The $L$-walk on $\Gamma_C$ corresponding to $p$.}
\label{generic_L_walk}
\end{center}
We examine the reflections of the input sequence, $\overline{s}=(\overline{d}_{0},\overline{d}_{\sigma(0)})\cdots (\overline{d}_{k},\overline{d}_{\sigma(k)})$, and the state-transition sequence,
\[
\overline{S}=\{(n-1,\overline{c}_{1}),(\overline{c}_{1},\overline{c}_{2}),\ldots,(\overline{c}_{j-1},0),(0,\overline{c}_{j+1}),\ldots,(\overline{c}_{k-1},\overline{c}_{k}),(\overline{c}_{k},n-1)\},
\]
of $p$. Now, both of these describe a closed walk on $\Gamma$, pictured in Figure \ref{generic_L_walk_reflection}, but since the initial and final states are not zero, $\overline{s}$ is not a permutiple string.
\begin{center}
\begin{footnotesize}
\begin{tikzpicture}
\tikzset{edge/.style = {->,> = latex'}}
\tikzset{vertex/.style = {shape=circle,draw,minimum size=1.0em}}
[xscale=1, yscale=1, auto=left,every node/.style={circle,fill=blue!20}]
\node[vertex] (n0) at (0,0) {$n-1$};
\node[vertex] (n1) at (2.4,0) {$\overline{c}_1$};
\node[vertex] (n2) at (4.5,0) {$\overline{c}_2$};
\draw[edge] (n0) edge node[above] {$(\overline{d}_0,\overline{d}_{\sigma(0)})$} (n1);
\draw[edge] (n1) edge node[above] {$(\overline{d}_1,\overline{d}_{\sigma(1)})$} (n2);
\node (2) at (5.1,0) {$\ldots$};
\node[vertex] (n0) at (5.9,0) {$\overline{c}_{j-1}$};
\node[vertex] (n1) at (8.8,0) {$0$};
\node[vertex] (n2) at (11,0) {$\overline{c}_{j+1}$};
\draw[edge] (n0) edge node[above] {$(\overline{d}_{j-1},\overline{d}_{\sigma(j-1)})$} (n1);
\draw[edge] (n1) edge node[above] {$(\overline{d}_j,\overline{d}_{\sigma(j)})$} (n2);
\node (3) at (11.8,0) {$\ldots$};
\node[vertex] (n0) at (12.5,0) {$\overline{c}_k$};
\node[vertex] (n1) at (14.9,0) {$n-1$};
\draw[edge] (n0) edge node[above] {$(\overline{d}_k,\overline{d}_{\sigma(k)})$} (n1);
\end{tikzpicture}
\end{footnotesize}
\captionof{figure}{The reflection of the $L$-walk  of $p$ on $\overline{\Gamma}_C$ (not an $L$-walk).}
\label{generic_L_walk_reflection}
\end{center}
However, a cyclic permutation of the sequence $\overline{S}$ does yield an $L$-walk. That is, $\overline{s}_{\psi^j}=(\overline{d}_{\psi^j(0)},\overline{d}_{\sigma\psi^j(0)})\cdots (\overline{d}_{\psi^j(k)},\overline{d}_{\sigma\psi^j(k)})$, is a permutiple string with the state-transition sequence $
\overline{S}_{\psi^j}=\{(0,\overline{c}_{j+1}),\ldots,(\overline{c}_{k-1},\overline{c}_{k}),(\overline{c}_{k},n-1),(n-1,\overline{c}_{1}),(\overline{c}_{1},\overline{c}_{2}),\ldots,(\overline{c}_{j-1},0)\}.$ We write $\overline{p}_{\psi^j}$ to denote the associated permutiple, $(\overline{d}_{\psi^j(k)},\ldots, \overline{d}_{\psi^j(0)})_b=n\cdot(\overline{d}_{\sigma\psi^j(k)},\ldots, \overline{d}_{\sigma\psi^j(0)})_b$. Now, since the inputs in the string $s$ make up the edges of $G_C$, we know that the collection of inputs of $\overline{s}_{\psi^j}$ make up the edges of $\overline{G}_C$. This is to say that $\overline{G}_C$ is the graph of a permutiple class since $\overline{G}_C=G_{\overline{p}_{\psi^j}}$.
\end{proof}

\begin{example}\label{example_1}
Consider the $(4,10)$-permutiple $p=(8,6,7,1,2)_{10}=4\cdot (2,1,6,7,8)_{10}$ and the class, $C$, with graph $G_C=G_p$. Figure \ref{reflection_example} shows $G_C$ and its reflection, $\overline{G}_C$.

\begin{center}
\begin{tikzpicture}
\tikzset{edge/.style = {->,> = latex'}}
\tikzset{vertex/.style = {shape=circle,draw,minimum size=1.5em}}
[xscale=2, yscale=2, auto=left,every node/.style={circle,fill=blue!20}]
\node[vertex] (n0) at (6,5) {$0$};
\node[vertex] (n1) at (5.618,6.17557) {$1$};
\node[vertex] (n2) at (4.618,6.90211) {$2$};
\node[vertex] (n3) at (3.38197,6.90211) {$3$};
\node[vertex] (n4) at (2.38197,6.17557) {$4$};
\node[vertex] (n5) at (2,5.00001) {$5$};
\node[vertex] (n6) at (2.38196,3.82443) {$6$};
\node[vertex] (n7) at (3.38196,3.09789) {$7$};
\node[vertex] (n8) at (4.618,3.09788) {$8$};
\node[vertex] (n9) at (5.618,3.82442) {$9$};
\draw[edge, bend right=5] (n8) to (n2);
\draw[edge, bend right=0] (n6) to (n1);
\draw[edge, bend right=0] (n7) to (n6);
\draw[edge, bend right=0] (n1) to (n7);
\draw[edge, bend right=5] (n2) to (n8);
\tikzset{edge/.style = {->,> = latex'}}
\tikzset{vertex/.style = {shape=circle,draw,minimum size=1.5em}}
[xscale=2, yscale=2, auto=left,every node/.style={circle,fill=blue!20}]
\node[vertex] (n0) at (12,5) {$0$};
\node[vertex] (n1) at (11.618,6.17557) {$1$};
\node[vertex] (n2) at (10.618,6.90211) {$2$};
\node[vertex] (n3) at (9.38197,6.90211) {$3$};
\node[vertex] (n4) at (8.38197,6.17557) {$4$};
\node[vertex] (n5) at (8,5.00001) {$5$};
\node[vertex] (n6) at (8.38196,3.82443) {$6$};
\node[vertex] (n7) at (9.38196,3.09789) {$7$};
\node[vertex] (n8) at (10.618,3.09788) {$8$};
\node[vertex] (n9) at (11.618,3.82442) {$9$};
\draw[edge, bend right=0] (n8) to (n2);
\draw[edge, bend right=5] (n7) to (n1);
\draw[edge, bend right=5] (n1) to (n7);
\draw[edge, bend right=0] (n3) to (n8);
\draw[edge, bend right=0] (n2) to (n3);
\end{tikzpicture}
\captionof{figure}{The permutiple graph $G_C$ (left) and its reflection $\overline{G}_C$ (right).}
\label{reflection_example}
\end{center}

We examine the correspondence between the cycles of $G_C$ and their cycle images as seen in Table \ref{reflection_example_2}. Since $\Gamma_C$ contains the vertex $n-1=3$, we may apply Theorem \ref{equiv_reflection} and conclude that $\overline{G}_C$ is the graph of some permutiple class. We may confirm this by examining the cycles of $\overline{G}_C$ and their corresponding cycle images shown in Table \ref{reflection_example_3}.

\begin{center}
\begin{tabular}{|c|c|l|c|}
\hline
&
{\bf Cycle of $G_C$} & {\bf Cycle Image}
&
\\\hline
$C_0$
&
\begin{tikzpicture}
\tikzset{edge/.style = {->,> = latex'}}
\tikzset{vertex/.style = {shape=circle,draw,minimum size=1.5em}}
[xscale=2, yscale=2, auto=left,every node/.style={circle,fill=blue!20}]
\node[vertex] (n1) at (4.5,5) {$1$};
\node[vertex] (n6) at (2.38196,3.82443) {$6$};
\node[vertex] (n7) at (3.38196,3.09789) {$7$};
\draw[edge, bend right=0] (n6) to (n1);
\draw[edge, bend right=0] (n7) to (n6);
\draw[edge, bend right=0] (n1) to (n7);
\end{tikzpicture}
&
\begin{tikzpicture}
\tikzset{edge/.style = {->,> = latex'}}
\tikzset{vertex/.style = {shape=circle,draw,minimum size=1.5em}}
[xscale=2, yscale=2, auto=left,every node/.style={circle,fill=blue!20}]
\node[vertex,initial,accepting] (n0) at (0,5) {$0$};
\node[vertex] (n2) at (3,5) {$2$};
\node[vertex] (n3) at (4.5,5) {$3$};
\draw[edge](n3) to[in=90,out=30, loop, style={min distance=10mm}] node[above] {\footnotesize{$(1,7)$}} (n3);
\draw[edge, bend right=10](n3) edge node[above] {\footnotesize{$(7,6)$}} (n2);
\draw[edge, bend right=10](n2) edge node[above] {\footnotesize{$(6,1)$}} (n0);
\end{tikzpicture}
&
$\Gamma_0$
\\\hline
$C_1$
&
\begin{tikzpicture}
\tikzset{edge/.style = {->,> = latex'}}
\tikzset{vertex/.style = {shape=circle,draw,minimum size=1.5em}}
[xscale=2, yscale=2, auto=left,every node/.style={circle,fill=blue!20}]
\node[vertex] (n2) at (4,5) {$2$};
\node[vertex] (n8) at (4,3) {$8$};
\draw[edge, bend right=10] (n8) to (n2);
\draw[edge, bend right=10] (n2) to (n8);
\end{tikzpicture}
&
\begin{tikzpicture}
\tikzset{edge/.style = {->,> = latex'}}
\tikzset{vertex/.style = {shape=circle,draw,minimum size=1.5em}}
[xscale=2, yscale=2, auto=left,every node/.style={circle,fill=blue!20}]
\node[vertex,initial,accepting] (n0) at (0,5) {$0$};
\node[vertex] (n3) at (4.5,5) {$3$};
 \draw[edge](n0) to[in=150,out=90, loop, style={min distance=10mm}] node[above] {\footnotesize{$(8,2)$}} (n0);
\draw[edge, bend left=20](n0) edge node[above] {\footnotesize{$(2,8)$}} (n3);
\end{tikzpicture}
&
$\Gamma_1$
\\\hline
\end{tabular}
\captionof{table}{The cycles of $G_C$ and their corresponding cycle images.}
\label{reflection_example_2}
\end{center}
\begin{center}
\begin{tabular}{|c|c|l|c|}
\hline
&
{\bf Cycle of $\overline{G}_C$}
&
{\bf Cycle Image}
&
\\\hline
$\overline{C}_0$
&
\begin{tikzpicture}
\tikzset{edge/.style = {->,> = latex'}}
\tikzset{vertex/.style = {shape=circle,draw,minimum size=1.5em}}
[xscale=2, yscale=2, auto=left,every node/.style={circle,fill=blue!20}]
\node[vertex] (n2) at (10.618,6.90211) {$2$};
\node[vertex] (n3) at (9.38197,6.90211) {$3$};
\node[vertex] (n8) at (10.618,5) {$8$};
\draw[edge, bend right=0] (n8) to (n2);
\draw[edge, bend right=0] (n3) to (n8);
\draw[edge, bend right=0] (n2) to (n3);
\end{tikzpicture}
&
\begin{tikzpicture}
\tikzset{edge/.style = {->,> = latex'}}
\tikzset{vertex/.style = {shape=circle,draw,minimum size=1.5em}}
[xscale=2, yscale=2, auto=left,every node/.style={circle,fill=blue!20}]
\node[vertex] (n0) at (4.5,5) {$3$};
\node[vertex] (n2) at (1.5,5) {$1$};
\node[vertex,initial,accepting] (n3) at (0,5) {$0$};
\draw[edge](n3) to[in=150,out=90, loop, style={min distance=10mm}] node[above] {\footnotesize{$(8,2)$}} (n3);
\draw[edge, bend left=10](n3) edge node[above] {\footnotesize{$(2,3)$}} (n2);
\draw[edge, bend left=10](n2) edge node[above] {\footnotesize{$(3,8)$}} (n0);
\end{tikzpicture}
&
$\overline{\Gamma}_0$
\\\hline
$\overline{C}_1$
&
\begin{tikzpicture}
\tikzset{edge/.style = {->,> = latex'}}
\tikzset{vertex/.style = {shape=circle,draw,minimum size=1.5em}}
[xscale=2, yscale=2, auto=left,every node/.style={circle,fill=blue!20}]
\node[vertex] (n1) at (11,5) {$1$};
\node[vertex] (n7) at (9.38196,3.09789) {$7$};
\draw[edge, bend right=10] (n7) to (n1);
\draw[edge, bend right=10] (n1) to (n7);
\end{tikzpicture}
&
\begin{tikzpicture}
\tikzset{edge/.style = {->,> = latex'}}
\tikzset{vertex/.style = {shape=circle,draw,minimum size=1.5em}}
[xscale=2, yscale=2, auto=left,every node/.style={circle,fill=blue!20}]
\node[vertex] (n0) at (4.5,5) {$3$};
\node[vertex,initial,accepting] (n3) at (0,5) {$0$};
\draw[edge](n0) to[in=90,out=30, loop, style={min distance=10mm}] node[above] {\footnotesize{$(1,7)$}} (n0);
\draw[edge, bend right=20](n0) edge node[above] {\footnotesize{$(7,1)$}} (n3);
\end{tikzpicture}
&
$\overline{\Gamma}_1$
\\\hline
\end{tabular}
\captionof{table}{The cycles of $\overline{G}_C$ and their corresponding cycle images.}
\label{reflection_example_3}
\end{center}

The resulting cycle-image union, $\overline{\Gamma}_C=\overline{\Gamma}_0 \cup \overline{\Gamma}_1$, forms a strongly-connected graph (as guaranteed by Corollary \ref{hs_topology}) containing the zero state. From this graph, we may easily construct permutiple strings, such as $(2,3)(3,8)(1,7)(7,1)(8,2)$, which corresponds to the $(4,10)$-permutiple $(8,7,1,3,2)_{10}=4\cdot (2,1,7,8,3)_{10}$.
\end{example}

\section{Symmetric classes}

\begin{definition}
Let $C$ be an $(n,b)$-permutiple class, and let $G_C$ be its graph. If the state $n-1$ is a vertex of $\Gamma_C$, we call the $(n,b)$-permutiple class with graph $\overline{G}_C$, guaranteed by Theorem \ref{equiv_reflection}, the {\it reflection} of $C$, and denote it by $\overline{C}$. If $C=\overline{C}$, then we say that $C$ is a {\it symmetric class}. The {\it symmetric closure} of $C$, denoted by $\widehat{C}$, is the permutiple class with graph $G_C \cup \overline{G}_C$.
\end{definition}

\begin{lemma}\label{un_of_refl_is_refl_of un_2}
Let $\Gamma_0$ and $\Gamma_1$ be subgraphs of $\Gamma$. Then, $\overline{\Gamma_0\cup \Gamma_1}=\overline{\Gamma}_0 \cup \overline{\Gamma}_1$.
\end{lemma}
\begin{proof}
Let $c$ be a vertex of $\overline{\Gamma_0\cup \Gamma_1}$. Then, $\overline{c}$ is a vertex of $\Gamma_0\cup \Gamma_1$. Thus, $\overline{c}$ is a vertex of $\Gamma_0$ or $\Gamma_1$, from which it follows that $c$ is a vertex of $\overline{\Gamma}_0$ or  $\overline{\Gamma}_1$. The other containment is proved similarly. Moreover, the above argument is identical in form for edges and edge labels.
\end{proof}

\begin{corollary}\label{un_of_refl_is_refl_of un_3}
Let $\Gamma_0,\ldots, \Gamma_m$ be subgraphs of $\Gamma$. Then, $\displaystyle\overline{\bigcup_{j=0}^{m}\Gamma_j}=\bigcup_{j=0}^{m}\overline{\Gamma}_j$.
\end{corollary}

\begin{theorem}\label{symm_equiv}
Let $C$ be an $(n,b)$-permutiple class, and let $G_C$ be its graph. Also, suppose that the state $n-1$ is a vertex of $\Gamma_C$.
Then, the symmetric closure of $C$ is a symmetric class. Also, all of the following are equivalent:
\begin{enumerate}
\item $C$ is a symmetric class.
\item $C$ is the symmetric closure of itself.
\item $\Gamma_C$ is a symmetric subgraph of the $(n,b)$-Hoey-Sloane graph.
\item $G_C$ is a symmetric graph.
\end{enumerate}
\end{theorem}
\begin{proof}
By Corollary \ref{un_of_refl_is_refl_of un_1}, the reflection of $G_{\widehat{C}}=G_C\cup \overline{G}_C$ is itself. Thus,  $\widehat{C}$ is its own reflection. By definition, we conclude that $\widehat{C}$ is symmetric. To demonstrate the equivalence, we show that items {\it 1}, {\it 2}, and {\it 3} are all equivalent to item {\it 4}.

$({\it 1} \iff {\it 4})$. The permutiple class $C$ has graph $G_C$, and, by definition, the permutiple class $\overline{C}$ has graph $\overline{G}_C$. From this, the equivalence between the statements $C=\overline{C}$ and $G_C=\overline{G}_C$ is clear.

$({\it 2} \iff {\it 4})$. If $C=\widehat{C}$, then $G_C=G_{\widehat{C}}=G_C\cup \overline{G}_C$. By Corollary \ref{un_of_refl_is_refl_of un_1}, we then have that $\overline{G}_C=\overline{G_C\cup \overline{G}_C}=\overline{G}_C\cup G_C=G_C$. Conversely, if $G_C$ is a symmetric graph, then $G_C=\overline{G}_C=G_C \cup \overline{G}_C$. Thus, $C$ and its symmetric closure have the same graph, from which it follows, by definition, that $C$ is its own symmetric closure.

$({\it 3} \iff {\it 4})$. Suppose $G_C$ is a symmetric graph, and let $\mathscr{C}=\{C_0,\ldots,C_m\}$ be the collection of cycles of $G_C$, and let $\overline{\mathscr{C}}=\{\overline{C}_0,\ldots,\overline{C}_m\}$ be the collection of cycles of $\overline{G}_C$. Also, let $\mathscr{I}=\{\Gamma_0,\ldots,\Gamma_m\}$ be the collection of cycle images of $\mathscr{C}$, and let $\overline{\mathscr{I}}=\{\overline{\Gamma}_0,\ldots,\overline{\Gamma}_m\}$ be the collection of cycle images of $\overline{\mathscr{C}}$. Since $G_C=\overline{G}_C$, it follows that $\mathscr{C}=\overline{\mathscr{C}}$, from which we have $\mathscr{I}=\overline{\mathscr{I}}$. Therefore, by Corollary \ref{un_of_refl_is_refl_of un_3}, the union of cycle images of $G_C$, that is, $\Gamma_C=\bigcup_{j=0}^{m}\Gamma_j$, is a symmetric subgraph of $\Gamma$. Supposing now that $\Gamma_C$ is symmetric, let $(d_1,d_2)$ be an edge of $G_C$. Then, $(d_1,d_2)$ is an edge label of $\Gamma_C$. By our assumption, we may say that $(\overline{d}_1,\overline{d}_2)$ is also an edge label of $\Gamma_C$, which is only possible if $(\overline{d}_1,\overline{d}_2)$ is an edge of $G_C$. It follows that $(d_1,d_2)$ is also an edge of $\overline{G}_C$. A similar argument shows that the edges of $\overline{G}_C$ are contained in the edges of $G_C$. Therefore, we may conclude that $G_C=\overline{G}_C$.
\end{proof}

\begin{example}\label{example_2}
We again consider the $(4,10)$-permutiple $p=(8,6,7,1,2)_{10}=4\cdot (2,1,6,7,8)_{10}$ and the class $C$ with graph $G_C=G_p$, which is depicted in Figure \ref{reflection_example}. The graph, $G_{\widehat{C}}$, of the symmetric closure, $\widehat{C}$, is seen in Figure \ref{symm_closure_example} and is a symmetric graph as guaranteed by Theorem \ref{symm_equiv}.
\begin{center}
\begin{tikzpicture}
\tikzset{edge/.style = {->,> = latex'}}
\tikzset{vertex/.style = {shape=circle,draw,minimum size=1.5em}}
[xscale=2, yscale=2, auto=left,every node/.style={circle,fill=blue!20}]
\node[vertex] (n0) at (6,5) {$0$};
\node[vertex] (n1) at (5.618,6.17557) {$1$};
\node[vertex] (n2) at (4.618,6.90211) {$2$};
\node[vertex] (n3) at (3.38197,6.90211) {$3$};
\node[vertex] (n4) at (2.38197,6.17557) {$4$};
\node[vertex] (n5) at (2,5.00001) {$5$};
\node[vertex] (n6) at (2.38196,3.82443) {$6$};
\node[vertex] (n7) at (3.38196,3.09789) {$7$};
\node[vertex] (n8) at (4.618,3.09788) {$8$};
\node[vertex] (n9) at (5.618,3.82442) {$9$};
\draw[edge, bend right=5] (n8) to (n2);
\draw[edge, bend right=0] (n6) to (n1);
\draw[edge, bend right=0] (n7) to (n6);
\draw[edge, bend right=5] (n2) to (n8);
\draw[edge, bend right=5] (n7) to (n1);
\draw[edge, bend right=5] (n1) to (n7);
\draw[edge, bend right=0] (n3) to (n8);
\draw[edge, bend right=0] (n2) to (n3);
\end{tikzpicture}
\captionof{figure}{The graph, $G_{\widehat{C}}$, of the symmetric closure, $\widehat{C}$, of $C$.}
\label{symm_closure_example}
\end{center}

The cycle images of $\widehat{C}$ are $\Gamma_0$,  $\Gamma_1$,  $\overline{\Gamma}_0$, and $\overline{\Gamma}_1$, given in Example \ref{example_1}. The union of these, $\Gamma_{\widehat{C}}$, is shown in Figure \ref{symm_closure_example_hsg}. We note that the configuration of the vertices and edges is to make the literal reflective symmetry more obvious.

\begin{center}
\begin{tikzpicture}
\tikzset{edge/.style = {->,> = latex'}}
\tikzset{vertex/.style = {shape=circle,draw,minimum size=1.5em}}
[xscale=2, yscale=2, auto=left,every node/.style={circle,fill=blue!20}]
\node[vertex,initial,accepting] (n0) at (0,5) {$0$};
\node[vertex] (n1) at (2,5) {$1$};
\node[vertex] (n2) at (4,5) {$2$};
\node[vertex] (n3) at (6,5) {$3$};
\draw[edge](n0) to[in=150,out=90, loop, style={min distance=10mm}] node[above] {\footnotesize{$(8,2)$}} (n0);
\draw[edge](n3) to[in=90,out=30, loop, style={min distance=10mm}] node[above] {\footnotesize{$(1,7)$}} (n3);
\draw[edge, bend right=0](n3) edge node[below] {\footnotesize{$(7,6)$}} (n2);
\draw[edge, bend right=25](n2) edge node[above] {\footnotesize{$(6,1)$}} (n0);
\draw[edge, bend left=80](n0) edge node[above] {\footnotesize{$(2,8)$}} (n3);
\draw[edge, bend right=0](n0) edge node[below] {\footnotesize{$(2,3)$}} (n1);
\draw[edge, bend left=25](n1) edge node[above] {\footnotesize{$(3,8)$}} (n3);
\draw[edge, bend right=60](n3) edge node[below] {\footnotesize{$(7,1)$}} (n0);
\end{tikzpicture}
\captionof{figure}{The cycle-image union, $\Gamma_{\widehat{C}}=\Gamma_0 \cup \Gamma_1 \cup \overline{\Gamma}_0 \cup \overline{\Gamma}_1$, of the cycles of $G_{\widehat{C}}$.}
\label{symm_closure_example_hsg}
\end{center}

We see that both $(8,6,7,1,2)_{10}=4\cdot (2,1,6,7,8)_{10}$ and $(8,7,1,3,2)_{10}=4\cdot (2,1,7,8,3)_{10}$, given in Example \ref{example_1}, are members of the symmetric closure of $C$.
\end{example}

\section{New permutiples from old}

We begin this section with a corollary to the argument in the proof of Theorem \ref{equiv_reflection}.

\begin{corollary}\label{reflective_sibling_thm}
Let $(d_k,\ldots, d_0)_b=n\cdot(d_{\sigma(k)},\ldots, d_{\sigma(0)})_b$ be an $(n,b,\sigma)$-permutiple with carry sequence $c_k,\ldots,c_1,c_0=0$. Then, for every $0 < j \leq k$ such that $c_j=n-1$, we have that $(\overline{d}_{\psi^j(k)},\ldots, \overline{d}_{\psi^j(0)})_b=n\cdot(\overline{d}_{\sigma\psi^j(k)},\ldots, \overline{d}_{\sigma\psi^j(0)})_b$ is an $(n,b,\psi^{-j}\sigma\psi^j)$-permutiple with carry sequence $\overline{c}_{\psi^j(k)},\ldots,\overline{c}_{\psi^j(1)},\overline{c}_{\psi^j(0)}=0$.
\end{corollary}

The above prompts us to define some additional terminology.

\begin{definition}
Suppose $p=(d_k,\ldots, d_0)_b=n\cdot(d_{\sigma(k)},\ldots, d_{\sigma(0)})_b$ is an $(n,b,\sigma)$-permutiple with carry sequence $c_k,\ldots,c_1,c_0=0$, where $c_j=n-1$. Then, we call the permutiple $\overline{p}_{\psi^j}=(\overline{d}_{\psi^j(k)},\ldots, \overline{d}_{\psi^j(0)})_b=n\cdot(\overline{d}_{\sigma\psi^j(k)},\ldots, \overline{d}_{\sigma\psi^j(0)})_b$ a {\it reflective sibling} of $p$.
\end{definition}

\begin{remark}
We draw the reader's attention to the use of the indefinite article in the above definition; for each $0 < j \leq k$ for which $c_j=n-1$, $p$ has a reflective sibling.
\end{remark}

\begin{example}\label{example_3}
We consider the $(4,10)$-permutiple featured in Example \ref{example_1}, namely,
\[
p=(d_4,d_3,d_2,d_1,d_0)=(8,6,7,1,2)_{10}=4\cdot (2,1,6,7,8)_{10}.
\]
Since its carry vector, $(c_4,c_3,c_2,c_1,c_0)=(0,2,3,3,0)$, contains two values of $n-1=3$, Corollary \ref{reflective_sibling_thm} guarantees that $p$ has two reflective siblings.
For $c_1=3$,

\begin{align*}
\overline{p}_{\psi}&=(\overline{d}_{\psi(4)},\overline{d}_{\psi(3)},\overline{d}_{\psi(2)},\overline{d}_{\psi(1)},\overline{d}_{\psi(0)})_{10}\\
&=(\overline{d}_0,\overline{d}_4,\overline{d}_3,\overline{d}_2,\overline{d}_1)_{10}\\
&=(7,1,3,2,8)_{10}\\
&=4\cdot (1,7,8,3,2)_{10}\\
&=4\cdot (\overline{d}_4,\overline{d}_0,\overline{d}_1,\overline{d}_3,\overline{d}_2)_{10}\\
&=4\cdot (\overline{d}_{\sigma \psi(4)},\overline{d}_{\sigma \psi(3)},\overline{d}_{\sigma \psi(2)},\overline{d}_{\sigma \psi(1)},\overline{d}_{\sigma \psi(0)})_{10},
\end{align*}
and for $c_2=3$,
\begin{align*}
\overline{p}_{\psi^2}&= (\overline{d}_{\psi^2(4)},\overline{d}_{\psi^2(3)},\overline{d}_{\psi^2(2)},\overline{d}_{\psi^2(1)},\overline{d}_{\psi^2(0)})_{10}\\
&=(\overline{d}_1,\overline{d}_0,\overline{d}_4,\overline{d}_3,\overline{d}_2)_{10}\\
&=(8,7,1,3,2)_{10}\\
&=4\cdot (2,1,7,8,3)_{10}\\
&=4\cdot (\overline{d}_2,\overline{d}_4,\overline{d}_0,\overline{d}_1,\overline{d}_3)_{10}\\
&=4\cdot (\overline{d}_{\sigma \psi^2(4)},\overline{d}_{\sigma \psi^2(3)},\overline{d}_{\sigma \psi^2(2)},\overline{d}_{\sigma \psi^2(1)},\overline{d}_{\sigma \psi^2(0)})_{10},
\end{align*}
the latter of which, was also mentioned in Example \ref{example_1}.
\end{example}

From past work \cite{holt_3,holt_4}, we have also seen that permutiples with zero carries other than $c_0=0$ enable us to find new permutiples. Specifically, we restate a result from previous work \cite[Corollary 1]{holt_3}.

\begin{theorem}\label{rotational_sibling_thm}
Let $(d_k, d_{k-1},\ldots, d_0)_b$ be an $(n,b,\sigma)$-permutiple with carries $c_k, c_{k-1},\ldots,c_0$.
If $c_j=0$, then $(d_{\psi^j(k)}, d_{\psi^{j}(k-1)},\ldots, d_{\psi^j(1)}, d_{\psi^j(0)})_b$
is an $(n,b,\psi^{-j} \sigma \psi^j)$-permutiple with carries
$c_{\psi^j(k)}$, $c_{\psi^{j}(k-1)}$, $\ldots$ , $c_{\psi^j(1)}$, $c_{\psi^j(0)}=c_j=0$.
\end{theorem}

The above inspires the next two definitions.

\begin{definition}
Suppose $p=(d_k,\ldots, d_0)_b=n\cdot(d_{\sigma(k)},\ldots, d_{\sigma(0)})_b$ is an $(n,b,\sigma)$-permutiple with carry sequence $c_k,\ldots,c_1,c_0=0$, where $c_j=0$. Then, we call the permutiple $p_{\psi^j}=(d_{\psi^j(k)},\ldots, d_{\psi^j(0)})_b=n\cdot(d_{\sigma\psi^j(k)},\ldots, d_{\sigma\psi^j(0)})_b$ a {\it rotational sibling} of $p$.
\end{definition}

\begin{definition}
Suppose $p$ is an $(n,b)$-permutiple. The collection of reflective and rotational siblings of $p$ together are called the {\it dihedral siblings} of $p$.
\end{definition}

\begin{example}\label{example_4}
Once again, we consider the $(4,10)$-permutiple
\[
p=(d_4,d_3,d_2,d_1,d_0)=(8,6,7,1,2)_{10} = 4 \cdot (2,1,6,7,8)_{10}
\]
with carry vector $(c_4,c_3,c_2,c_1,c_0)=(0,2,3,3,0)$.  In Example \ref{example_3} we found the two reflective siblings of $p$. We now find its rotational siblings. In doing so, we complete the collection of the dihedral siblings of $p$.

By Theorem \ref{rotational_sibling_thm}, every permutiple is trivially its own rotational sibling since $c_0$ must always be zero. On the other hand, we have a nontrivial rotational sibling for $c_4=0$, namely,
\begin{align*}
p_{\psi^4}&= (d_{\psi^4(4)},d_{\psi^4(3)},d_{\psi^4(2)},d_{\psi^4(1)},d_{\psi^4(0)})_{10}\\
&=(d_{3},d_{2},d_{1},d_{0},d_{4})_{10}\\
&=(6,7,1,2,8)_{10}\\
&=4\cdot (1,6,7,8,2)_{10}\\
&=4\cdot (d_1,d_3,d_2,d_4,d_0)_{10}\\
&=4\cdot (d_{\sigma \psi^4(4)},d_{\sigma \psi^4(3)},d_{\sigma \psi^4(2)},d_{\sigma \psi^4(1)},d_{\sigma \psi^4(0)})_{10}.
 \end{align*}
We see that, altogether, $p$ has four dihedral siblings: $p$, $p_{\psi^4}$, $\overline{p}_{\psi}$, and $\overline{p}_{\psi^2}$.
\end{example}

It is here that we say more about our chosen terminology. Let $C$ be an $(n,b)$-permutiple class for which $\Gamma_C$ has the state  $n-1$ as a vertex. We consider its symmetric closure, $\widehat{C}$. Letting $\mathscr{I}$ be the collection of cycle images of the cycles of $G_{\widehat{C}}$, we have seen, in general, that the union of the elements of $\mathscr{I}$ is a symmetric subgraph, $\Gamma_{\widehat{C}}$, of the $(n,b)$-Hoey-Sloane graph. We now consider the collection of closed walks, $\mathscr{W}_{\ell}$, on $\Gamma_{\widehat{C}}$ which have a fixed length, $\ell$. To denote an element of $\mathscr{W}_{\ell}$, we use the notation $(c_0,c_1,\ldots,c_{\ell-1})$ to mean the state-transition sequence
\[
\{(c_{0},c_{1}),(c_{1},c_{2}),\ldots, (c_{\ell-2},c_{\ell-1}),(c_{\ell-1},c_{0})\}
\]
on $\Gamma_{\widehat{C}}$ and its corresponding edge-label sequence,
\[
\{(d_{0},\widehat{d}_{0}),(d_{1},\widehat{d}_{1}),\ldots, (d_{\ell-2},\widehat{d}_{\ell-2}),(d_{\ell-1},\widehat{d}_{\ell-1})\},
\]
where $(d_j,\widehat{d}_{j})$ induces the transition from $c_j$ to $c_{j+1}$. We use this notation regardless of whether the collection is a cycle or a circuit on $\Gamma_{\widehat{C}}$.

Under these conditions, the dihedral group, $D_{\ell}$, acts on $\mathscr{W}_{\ell}$ by cyclically permuting vertices, edges, and labels, as well as reflecting vertices, edges, and labels in the sense established by Definition \ref{hs_subgraph_reflection}. More specifically, letting $\alpha$ be a $360^{\circ}/\ell$ rotation and $\beta$ a reflection, we may define the action by how these generating elements act on $\mathscr{W}_{\ell}:$ $\alpha w=(c_{\psi(0)},c_{\psi(1)},\ldots,c_{\psi(\ell-1)})$, where $\psi$ is the $\ell$-cycle $(0,1,\ldots,\ell-1)$, and $\beta w=(\overline{c}_0,\overline{c}_1,\ldots,\overline{c}_{\ell-1})$.

In this setting, we may now say that two permutiples, $p_1$ and $p_2$, are dihedral siblings if there is a composition of dihedral transformations which acts on the $L$-walk, $w_1$, of $p_1$ to produce the $L$-walk, $w_2$, of $p_2$ (that is, if $w_1$ and $w_2$ are in the same orbit under the action of $D_{\ell}$).

To illustrate the above, we consider the dihedral siblings listed in Example \ref{example_4}. The $L$-walks of $p$ and $p_{\psi^4}$ are visualized in the shape of a regular polygon in Figure \ref{4_10_L_walk_1}.

\begin{center}
\begin{tikzpicture}
\tikzset{edge/.style = {->,> = latex'}}
\tikzset{vertex/.style = {shape=circle,draw,minimum size=1.5em}}
[xscale=2, yscale=2, auto=left,every node/.style={circle,fill=blue!20}]
\node[vertex, initial right] (n0) at (3.7,7) {$0$};
\node[vertex] (n1) at (2.318036007720638,8.9021123765857615) {$3$};
\node[vertex] (n2) at (0.081968506839264,8.17557393946108) {$3$}; \node[vertex] (n3) at (0.081962267880038,5.8244346477391877) {$2$}; \node[vertex] (n4) at (2.318025912859958,5.0978843434129413) {$0$};
\draw[edge] (n0) edge node[right] {$(2,8)$} (n1);
\draw[edge] (n1) edge node[above] {$(1,7)$} (n2);
\draw[edge] (n2) edge node[left] {$(7,6)$} (n3);
\draw[edge] (n3) edge node[below] {$(6,1)$} (n4);
\draw[edge] (n4) edge node[right] {$(8,2)$} (n0);
\node[vertex, initial right] (n0) at (11.0,7) {$0$};
\node[vertex] (n1) at (9.718036007720638,8.9021123765857615) {$0$};
\node[vertex] (n2) at (7.481968506839264,8.17557393946108) {$3$}; \node[vertex] (n3) at (7.481962267880038,5.8244346477391877) {$3$}; \node[vertex] (n4) at (9.718025912859958,5.0978843434129413) {$2$};
\draw[edge] (n0) edge node[right] {$(8,2)$} (n1);
\draw[edge] (n1) edge node[above] {$(2,8)$} (n2);
\draw[edge] (n2) edge node[left] {$(1,7)$} (n3);
\draw[edge] (n3) edge node[below] {$(7,6)$} (n4);
\draw[edge] (n4) edge node[right] {$(6,1)$} (n0);
\end{tikzpicture}
\captionof{figure}{The $L$-walks corresponding to $p=(8,6,7,1,2)_{10}=4\cdot (2,1,6,7,8)_{10}$ (left) and its rotational sibling, $p_{\psi^4}=(6,7,1,2,8)_{10}
=4\cdot (1,6,7,8,2)_{10}$ (right).}
\label{4_10_L_walk_1}
\end{center}

The reflection of the $L$-walk corresponding to $p$ is visualized in Figure \ref{4_10_L_walk_2}.
\begin{center}
\begin{tikzpicture}
\tikzset{edge/.style = {->,> = latex'}}
\tikzset{vertex/.style = {shape=circle,draw,minimum size=1.5em}}
[xscale=2, yscale=2, auto=left,every node/.style={circle,fill=blue!20}]
\node[vertex, initial right] (n0) at (3.7,2) {$3$};
\node[vertex] (n1) at (2.318036007720638,3.9021123765857615) {$0$};
\node[vertex] (n2) at (0.081968506839264,3.17557393946108) {$0$}; \node[vertex] (n3) at (0.081962267880038,0.8244346477391877) {$1$}; \node[vertex] (n4) at (2.318025912859958,0.0978843434129413) {$3$};
\draw[edge] (n0) edge node[right] {$(7,1)$} (n1);
\draw[edge] (n1) edge node[above] {$(8,2)$} (n2);
\draw[edge] (n2) edge node[left] {$(2,3)$} (n3);
\draw[edge] (n3) edge node[below] {$(3,8)$} (n4);
\draw[edge] (n4) edge node[right] {$(1,7)$} (n0);
\end{tikzpicture}
\captionof{figure}{The closed walk resulting from the reflection of the $L$-walk corresponding to $p$.}
\label{4_10_L_walk_2}
\end{center}

Although the result depicted in Figure \ref{4_10_L_walk_2} is not an $L$-walk, there are two rotations which do result in $L$-walks. These are visualized in Figure \ref{4_10_L_walk_3}, and correspond to the reflective siblings of $p$.

\begin{center}
\begin{tikzpicture}
\tikzset{edge/.style = {->,> = latex'}}
\tikzset{vertex/.style = {shape=circle,draw,minimum size=1.5em}}
[xscale=2, yscale=2, auto=left,every node/.style={circle,fill=blue!20}]
\node[vertex, initial right] (n0) at (3.7,2) {$0$};
\node[vertex] (n1) at (2.318036007720638,3.9021123765857615) {$0$};
\node[vertex] (n2) at (0.081968506839264,3.17557393946108) {$1$}; \node[vertex] (n3) at (0.081962267880038,0.8244346477391877) {$3$}; \node[vertex] (n4) at (2.318025912859958,0.0978843434129413) {$3$};
\draw[edge] (n0) edge node[right] {$(8,2)$} (n1);
\draw[edge] (n1) edge node[above] {$(2,3)$} (n2);
\draw[edge] (n2) edge node[left] {$(3,8)$} (n3);
\draw[edge] (n3) edge node[below] {$(1,7)$} (n4);
\draw[edge] (n4) edge node[right] {$(7,1)$} (n0);
\node[vertex, initial right] (n0) at (11.0,2) {$0$};
\node[vertex] (n1) at (9.718036007720638,3.9021123765857615) {$1$};
\node[vertex] (n2) at (7.481968506839264,3.17557393946108) {$3$}; \node[vertex] (n3) at (7.481962267880038,0.8244346477391877) {$3$}; \node[vertex] (n4) at (9.718025912859958,0.0978843434129413) {$0$};
\draw[edge] (n0) edge node[right] {$(2,3)$} (n1);
\draw[edge] (n1) edge node[above] {$(3,8)$} (n2);
\draw[edge] (n2) edge node[left] {$(1,7)$} (n3);
\draw[edge] (n3) edge node[below] {$(7,1)$} (n4);
\draw[edge] (n4) edge node[right] {$(8,2)$} (n0);
\end{tikzpicture}
\captionof{figure}{The $L$-walks corresponding to the reflective siblings of $p$: $\overline{p}_{\psi}=(7,1,3,2,8)_{10}=4\cdot (1,7,8,3,2)_{10}$ (left) and $\overline{p}_{\psi^2}=(8,7,1,3,2)_{10}=4\cdot (2,1,7,8,3)_{10}$ (right).}
\label{4_10_L_walk_3}
\end{center}

Here we add that when the configuration of vertices and edges highlights a literal line of symmetry of $\Gamma_C$ inherited from the mother graph, as is the case in Figure \ref{symm_closure_example_hsg}, the reflection of an $L$-walk results in a walk on $\Gamma_C$ which is a literal reflection over this line of symmetry. Rotations acting on such walks result in $L$-walks when the zero state occupies the initial position as seen above. These represent the reflective siblings of some initial permutiple. The above helps to explain not only the nomenclature adopted by this effort, but also, perhaps more importantly, the types of permutations which arise when finding new permutiples of a fixed length from old \cite{holt_3,holt_4} (see Table \ref{conj_class_table}).

The above observations also demonstrate that an $(n,b)$-permutiple, $p$, having at least one carry equal to $n-1$ has a reflective sibling consisting of the reflected and cyclically-permuted digits of $p$.  If $p$ belongs to a symmetric class, $C$, then every reflective sibling of $p$ is still a member of $C$. More specifically, when the reflected digits form the same multiset as the digits of the original example, we can say more.

\begin{corollary}\label{sym_rev}
Suppose that $p=(d_{\pi(k)},\ldots, d_{\pi(0)})_b=n\cdot(d_{\pi \sigma(k)},\ldots, d_{\pi \sigma(0)})_b$ is an $(n,b,\sigma)$-permutiple whose $j^{th}$ carry is $n-1$. If $\{d_0,d_1,\ldots, d_k\}$ and $\{\overline{d}_0,\overline{d}_1,\ldots, \overline{d}_k\}$  are the same multiset, where $d_0 \leq d_1 \leq  \cdots \leq d_k$, then the reflective sibling of $p$,
\[
\overline{p}_{\psi^j}=(\overline{d}_{\pi \psi^j(k)},\ldots, \overline{d}_{\pi \psi^j(0)})_b=n \cdot (\overline{d}_{\pi \sigma \psi^j(k)},\ldots, \overline{d}_{\pi \sigma \psi^j(0)})_b,
\]
may also be represented as
\[
\overline{p}_{\psi^j}=(d_{\rho \pi \psi^j(k)},\ldots, d_{\rho \pi \psi^j(0)})_b=n\cdot(d_{\rho \pi \sigma\psi^j(k)},\ldots, d_{\rho \pi \sigma\psi^j(0)})_b,
\]
where $\rho$ is the reversal permutation.
\end{corollary}

\begin{proof}
Since the reflected digits of $p$ yield the same multiset, we may say that  $\overline{d}_0=d_{\rho(0)}\geq \overline{d}_1=d_{\rho(1)}\geq \cdots \geq \overline{d}_k=d_{\rho(k)}$. Since $\overline{d}_j=d_{\rho(j)}$ for all $0 \leq j \leq k$,  the statement follows.
\end{proof}

Given an $(n,b)$-permutiple, $p=(d_k,\ldots,d_0)_b$, whose $j^{th}$ carry is $n-1$, it is not difficult to manufacture examples which allow for the application of Corollary \ref{sym_rev}. Let $C$ be the permutiple class whose graph is $G_C=G_p$. We consider the reflective sibling, $\overline{p}_{\psi^j}=(\overline{d}_{\psi^j(k)},\ldots,\overline{d}_{\psi^j(0)})_b$, of $p$, and its class, $\overline{C}$, with graph $\overline{G}_C=G_{\overline{p}_{\psi^j}}$. The multiset union of the digits of $p$ and $\overline{p}_{\psi^j}$ give us a collection of digits whose reflection is itself. Thus, the concatenation of the digit string of $p$ and $\overline{p}_{\psi^j}$ give us an $(n,b)$-permutiple, $\widehat{p}$, which is a member of the symmetric closure of $C$, and to which we may apply Corollary \ref{sym_rev}. Of course, generally speaking, $\widehat{p}$ is only one of many members of $\widehat{C}$ with the same digits. We may easily find other examples by using the multi-image union, $\Delta_I$, corresponding to the multiset union of mother-graph cycles, $C_I$, which forms the collection of inputs constituting the permutiple string of $\widehat{p}$. Counting such examples becomes a problem of counting muli-Eulerian circuits of a directed graph, which is a problem taken up by Farrell and Levine \cite{farrell}.

The next example illustrates the above points, verifies Corollary \ref{sym_rev}, and  completes our consideration of the class of $(4,10)$-permutiples featured in Examples \ref{example_1} through \ref{example_4}.

\begin{example}\label{example_7}
Consider the $(4,10)$-permutiple
$p=(8,6,7,1,2)_{10}=4\cdot (2,1,6,7,8)_{10}$,  whose second carry is $n-1=3$, and its reflective sibling, $\overline{p}_{\psi^2}=(8,7,1,3,2)_{10}=4\cdot (2,1,7,8,3)_{10}$, given in Example \ref{example_3}. A concatenation of these two examples gives an element,
\[
\widehat{p}=(8,6,7,1,2,8,7,1,3,2)_{10}=4 \cdot       (2,1,6,7,8,2,1,7,8,3)_{10},
\]
of $\widehat{C}$ whose corresponding permutiple string is
\[
 \widehat{s}=(2,3)(3,8)(1,7)(7,1)(8,2)(2,8)(1,7)(7,6)(6,1)(8,2).
\]
The inputs of $\widehat{s}$ make up the multiset union of mother-graph cycles for this example, $C_I=C_0 \uplus C_1 \uplus \overline{C}_0 \uplus \overline{C}_1$, and the constituent cycles are shown in Tables \ref{reflection_example_2} and \ref{reflection_example_3}. The corresponding multi-images are $\Delta_0$, $\Delta_1$, $\overline{\Delta}_0$, and $\overline{\Delta}_1$ (where multigraph reflection is defined analogously as in Definition \ref{hs_subgraph_reflection}), and these multi-graphs look identical to  $\Gamma_0$, $\Gamma_1$, $\overline{\Gamma}_0$, and $\overline{\Gamma}_1$  (also shown in Tables \ref{reflection_example_2} and \ref{reflection_example_3}). The multi-image union $\Delta_I=\Delta_0 \uplus \Delta_1 \uplus \overline{\Delta}_0 \uplus \overline{\Delta}_1$ corresponding to $C_I$ is shown in Figure \ref{symm_closure_example_multi_hsg}, and the reader may verify that $\Delta_I$ satisfies the conditions of Corollary \ref{indegree_outdegree}. The reader may also observe that $\Delta_I$ differs from the set-theoretic union, $\Gamma_C$, shown in Figure \ref{symm_closure_example_hsg}.

\begin{center}
\begin{tikzpicture}
\tikzset{edge/.style = {->,> = latex'}}
\tikzset{vertex/.style = {shape=circle,draw,minimum size=1.5em}}
[xscale=2, yscale=2, auto=left,every node/.style={circle,fill=blue!20}]
\node[vertex,initial,accepting] (n0) at (0,5) {$0$};
\node[vertex] (n1) at (2,5) {$1$};
\node[vertex] (n2) at (4,5) {$2$};
\node[vertex] (n3) at (6,5) {$3$};
\draw[edge](n0) to[in=150,out=90, loop, style={min distance=10mm}] node[above] {\footnotesize{$(8,2)$}} (n0);
\draw[edge](n0) to[in=160,out=80, loop, style={min distance=25mm}] node[above] {\footnotesize{$(8,2)$}} (n0);
\draw[edge](n3) to[in=80,out=30, loop, style={min distance=10mm}] node[above] {\footnotesize{$(1,7)$}} (n3);
\draw[edge](n3) to[in=90,out=20, loop, style={min distance=25mm}] node[above] {\footnotesize{$(1,7)$}} (n3);
\draw[edge, bend right=0](n3) edge node[below] {\footnotesize{$(7,6)$}} (n2);
\draw[edge, bend right=25](n2) edge node[above] {\footnotesize{$(6,1)$}} (n0);
\draw[edge, bend left=70](n0) edge node[above] {\footnotesize{$(2,8)$}} (n3);
\draw[edge, bend right=0](n0) edge node[below] {\footnotesize{$(2,3)$}} (n1);
\draw[edge, bend left=25](n1) edge node[above] {\footnotesize{$(3,8)$}} (n3);
\draw[edge, bend right=60](n3) edge node[below] {\footnotesize{$(7,1)$}} (n0);
\end{tikzpicture}
\captionof{figure}{The multi-image union, $\Delta_I=\Delta_0 \uplus \Delta_1 \uplus \overline{\Delta}_0 \uplus \overline{\Delta}_1$, corresponding to the multiset union $C_I=C_0 \uplus C_1 \uplus \overline{C}_0 \uplus \overline{C}_1$ of the cycles of $G_{\widehat{C}}$.}
\label{symm_closure_example_multi_hsg}
\end{center}

As mentioned more generally, we may use $\Delta_I$ in Figure \ref{symm_closure_example_multi_hsg} to create new orderings of $C_I$ which produce permutiple strings. For instance, the input string
\[
s=(2,3)(3,8)(1,7)(1,7)(7,6)(6,1)(8,2)(8,2)(2,8)(7,1)
\]
determines an Eulerian circuit on $\Delta_I$ from the zero state back to itself (this is precisely what the conditions of Corollary \ref{indegree_outdegree} guarantee \cite{bang,farrell}). This is to say that $s$ is a permutiple string.

We now verify Corollary \ref{sym_rev} by considering the $(4,10)$-permutiple corresponding to $s$,
\[
q=(7,2,8,8,6,7,1,1,3,2)_{10}=4\cdot(1,8,2,2,1,6,7,7,8,3)_{10},
\]
with carry vector $(c_9,c_8,c_7,c_6,c_5,c_4,c_3,c_2,c_0)=(3,0,0,0,2,3,3,3,1,0)$. The ordered multiset of digits of both $\widehat{p}$ and $q$ is
\[
\{d_0,d_1,d_2,d_3,d_4,d_5,d_6,d_7,d_8,d_9\}=\{1,1,2,2,3,6,7,7,8,8\},
\]
and clearly $\overline{d}_j=d_{\rho(j)}$ for all $0 \leq j \leq 9$, where $\rho$ is the reversal permutation, $\rho(j)=9-j$.
Choose $\pi=\left(
\begin{array}{cccccccccc}
0 & 1 & 2 & 3 & 4 & 5 & 6 & 7 & 8 & 9\\
2 & 4 & 0 & 1 & 6 & 5 & 8 & 9 & 3 & 7\\
\end{array}
\right)$ and $\sigma=\left(
\begin{array}{cccccccccc}
0 & 1 & 2 & 3 & 4 & 5 & 6 & 7 & 8 & 9\\
1 & 6 & 4 & 9 & 5 & 2 & 0 & 8 & 7 & 3\\
\end{array}
\right)$
 so that
\begin{align*}
q&=(d_{\pi(9)},d_{\pi(8)},d_{\pi(7)},d_{\pi(6)},d_{\pi(5)},d_{\pi(4)},d_{\pi(3)},d_{\pi(2)},d_{\pi(1)},d_{\pi(0)})_{10}\\
&=(d_7,d_3,d_9,d_8,d_5,d_6,d_1,d_0,d_4,d_2)_{10}\\
&=(7,2,8,8,6,7,1,1,3,2)_{10}\\
&=4\cdot (1,8,2,2,1,6,7,7,8,3)_{10}\\
&=4\cdot (d_1,d_9,d_3,d_2,d_0,d_5,d_7,d_6,d_8,d_4)_{10}\\
&=4\cdot (d_{\pi\sigma(9)},d_{\pi\sigma(8)},d_{\pi\sigma(7)},d_{\pi\sigma(6)},d_{\pi\sigma(5)},d_{\pi\sigma(4)},d_{\pi\sigma(3)},d_{\pi\sigma(2)},d_{\pi\sigma(1)},d_{\pi\sigma(0)})_{10}.
\end{align*}
For $c_2=3$, we have a reflective sibling,
\begin{align*}
\overline{q}_{\psi^2}&=(\overline{d}_{\pi \psi^2(9)},\overline{d}_{\pi \psi^2(8)},\overline{d}_{\pi \psi^2(7)},\overline{d}_{\pi \psi^2(6)},\overline{d}_{\pi \psi^2(5)},\overline{d}_{\pi \psi^2(4)},\overline{d}_{\pi \psi^2(3)},\overline{d}_{\pi \psi^2(2)},\overline{d}_{ \pi \psi^2(1)}, \overline{d}_{\pi \psi^2(0)})_{10}\\
&=(\overline{d}_4,\overline{d}_2,\overline{d}_7,\overline{d}_3,\overline{d}_9,\overline{d}_8,\overline{d}_5,\overline{d}_6,\overline{d}_1,\overline{d}_0)_{10}\\
&=(\overline{3},\overline{2},\overline{7},\overline{2},\overline{8},\overline{8},\overline{6},\overline{7},\overline{1},\overline{1})_{10}\\
&=(6,7,2,7,1,1,3,2,8,8)_{10}.
\end{align*}
With all of its hypotheses met, we may now invoke Corollary \ref{sym_rev}:
\begin{align*}
\overline{q}_{\psi^2}&=(d_{\rho \pi \psi^2(9)},d_{\rho \pi \psi^2(8)},\ldots,d_{\rho \pi \psi^2(1)}, d_{\rho \pi \psi^2(0)})_{10}\\
&=(d_5,d_7,d_2,d_6,d_0,d_1,d_4,d_3,d_8,d_9)_{10}\\
&=(6,7,2,7,1,1,3,2,8,8)_{10}\\
&=4\cdot(1,6,8,1,7,7,8,3,2,2)_{10}\\
&=4\cdot(d_1,d_5,d_8,d_0,d_6,d_7,d_9,d_4,d_2,d_3)_{10}\\
&=4\cdot(d_{\rho \pi \sigma\psi^2(9)},d_{\rho \pi \sigma\psi^2(8)},\ldots, d_{\rho \pi \sigma\psi^2(1)}, d_{\rho \pi \sigma\psi^2(0)})_{10}.
\end{align*}
\end{example}

As the above illustrates, Corollary \ref{sym_rev} implicitly assumes that a permutiple, $p$, resides within a symmetric class, $C$.

To bring together several of the ideas we have discussed so far, and to showcase the generality of these methods, we now consider a base-$4$ example.

\begin{example}\label{example_5}
Consider the $(3,4)$-permutiple $p=(3,1,1,0,2,2)_4=3\cdot(1,0,1,2,3,2)_4$ with carry vector $(c_5,c_4,c_3,c_2,c_1,c_0)=(0,1,2,2,1,0)$. Let $C$ be the permutiple class determined by the graph of $p$. The $(3,4)$-mother graph is seen in Figure \ref{3_4_graph} with the edges of $G_C=G_p$ highlighted in bold red. The cycles and cycle images of $G_C$ are given in Table \ref{3_4_ci}.
 \begin{center}
\begin{tikzpicture}
\tikzset{edge/.style = {->,> = latex'}} \tikzset{vertex/.style = {shape=circle,draw,minimum size=1.5em}} [xscale=2, yscale=2, auto=left,every node/.style={circle,fill=blue!20}] \node[vertex] (n0) at (21.5,5) {$0$}; \node[vertex] (n1) at (20,6.5) {$1$}; \node[vertex] (n2) at (18.5,5) {$2$}; \node[vertex] (n3) at (20,3.5) {$3$};
\draw[edge, bend right=10, red, very thick](n0) to (n2);
\draw[edge, bend right=10](n2) to (n0);
\draw[edge, bend right=10, red, very thick](n1) to (n0);
\draw[edge, bend right=10](n0) to (n1);
\draw[edge, red, very thick](n1) to[in=130,out=50, loop, style={min distance=10mm}] (n1);
\draw[edge, red, very thick](n2) to[in=215,out=135, loop, style={min distance=10mm}] (n2);
\draw[edge](n0) to[in=40,out=-40, loop, style={min distance=10mm}] (n0);
\draw[edge](n3) to[in=310,out=230, loop, style={min distance=10mm}] (n3);
\draw[edge, bend right=10, red, very thick](n2) to (n3);
\draw[edge, bend right=10](n3) to (n2);
\draw[edge, bend right=10, red, very thick](n3) to (n1);
\draw[edge, bend right=10](n1) to (n3);
\end{tikzpicture}
\captionof{figure}{The $(3,4)$-mother graph with the edges of $G_C=G_p$ featured in bold red.}
\label{3_4_graph}.
\end{center}
\begin{center}
\begin{tabular}{|c|c|l|c|}
\hline
&
{\bf Cycle of $G_C$}
&
{\bf Cycle Image}
&
\\\hline
$C_1$
&
\begin{tikzpicture} \tikzset{edge/.style = {->,> = latex'}} \tikzset{vertex/.style = {shape=circle,draw,minimum size=1.5em}} [xscale=2, yscale=2, auto=left,every node/.style={circle,fill=blue!20}]
\node[vertex] (n1) at (6,6) {$1$};
\draw[edge](n1) to[loop] (n1);
\end{tikzpicture}
&
\begin{tikzpicture} \tikzset{edge/.style = {->,> = latex'}} \tikzset{vertex/.style = {shape=circle,draw,minimum size=1.5em}} [xscale=2, yscale=2, auto=left,every node/.style={circle,fill=blue!20}]
\color{white}\node[vertex,initial,accepting] (n0) at (0,5) {$0$};
\node[vertex,black] (n1) at (2,5) {$1$};
\node[vertex,black] (n2) at (4,5) {$2$};
 \draw[edge,black,bend left=10](n2) edge node[below] {\footnotesize $(1,1)$} (n1);
\end{tikzpicture}
&
$\Gamma_1$
\\\hline
$C_2$
&
\begin{tikzpicture} \tikzset{edge/.style = {->,> = latex'}} \tikzset{vertex/.style = {shape=circle,draw,minimum size=1.5em}} [xscale=2, yscale=2, auto=left,every node/.style={circle,fill=blue!20}]
\node[vertex] (n2) at (5,5) {$2$};
 \draw[edge](n2) to[in=220,out=130, loop, style={min distance=15mm}] (n2);
\end{tikzpicture}
&
\begin{tikzpicture} \tikzset{edge/.style = {->,> = latex'}} \tikzset{vertex/.style = {shape=circle,draw,minimum size=1.5em}} [xscale=2, yscale=2, auto=left,every node/.style={circle,fill=blue!20}]
\node[vertex,initial,accepting] (n0) at (0,5) {$0$};
\node[vertex] (n1) at (2,5) {$1$};
\node[vertex,white] (n2) at (4,5) {$2$};
 \draw[edge, bend right=10](n0) edge node[below] {\footnotesize $(2,2)$} (n1);
\end{tikzpicture}
&
$\Gamma_2$
\\\hline
$C_3$
&
\begin{tikzpicture} \tikzset{edge/.style = {->,> = latex'}} \tikzset{vertex/.style = {shape=circle,draw,minimum size=1.5em}} [xscale=2, yscale=2, auto=left,every node/.style={circle,fill=blue!20}]
\node[vertex] (n0) at (3,2) {$0$};
\node[vertex] (n1) at (1.5,3.5) {$1$};
\node[vertex] (n2) at (0,2) {$2$};
\node[vertex] (n3) at (1.5,0.5) {$3$};
\draw[edge](n0) to (n2);
\draw[edge](n1) to (n0);
\draw[edge](n2) to (n3);
\draw[edge](n3) to (n1);
\end{tikzpicture}
&
\begin{tikzpicture} \tikzset{edge/.style = {->,> = latex'}} \tikzset{vertex/.style = {shape=circle,draw,minimum size=1.5em}} [xscale=2, yscale=2, auto=left,every node/.style={circle,fill=blue!20}]
\node[vertex,initial,accepting] (n0) at (0,5) {$0$};
\node[vertex] (n1) at (2,5) {$1$};
\node[vertex] (n2) at (4,5) {$2$};
\draw[edge](n0) to[in=160,out=100, loop, style={min distance=10mm}] node[above] {\footnotesize $(3,1)$} (n0);
\draw[edge](n2) to[in=80,out=20, loop, style={min distance=10mm}] node[above] {\footnotesize $(0,2)$} (n2);
\draw[edge, bend right=10](n1) edge node[above] {\footnotesize $(1,0)$} (n0);
\draw[edge, bend left=10](n1) edge node[above] {\footnotesize $(2,3)$} (n2);
\end{tikzpicture}
&
$\Gamma_3$
\\\hline
\end{tabular}
\captionof{table}{The cycles of $G_C$ and their corresponding cycle images.}
\label{3_4_ci}
\end{center}

The union of the cycle images, $\Gamma_C=\Gamma_1 \cup \Gamma_2 \cup \Gamma_3$, is a labeled subgraph of the $(3,4)$-Hoey-Sloane graph, and is shown in Figure \ref{3_4_hsg} with the labeled edges of $\Gamma_C$ highlighted in bold red. Also, the configuration of the vertices and edges emphasizes a literal vertical line of symmetry of both $\Gamma$ and $\Gamma_C$ inherited from the mother graph. The permutiple string, $s=(2,2)(2,3)(0,2)(1,1)(1,0)(3,1)$, of $p$ determines the $L$-walk on $\Gamma_C$.

\begin{center}
\begin{tikzpicture}
\tikzset{edge/.style = {->,> = latex'}}
\tikzset{vertex/.style = {shape=circle,draw,minimum size=1.5em}}
[xscale=2, yscale=2, auto=left,every node/.style={circle,fill=blue!20}]
\node[vertex,initial,accepting] (n0) at (0,5) {$\bf 0$};
\node[vertex] (n1) at (3.2,5) {$\bf 1$};
\node[vertex] (n2) at (6.4,5) {$\bf 2$};
\draw[edge,red,very thick](n0) to[in=160,out=100, loop, style={min distance=10mm}] node[left] {\footnotesize{$\color{black}{(0,0),}\color{red}{\bf (3,1)}$}} (n0);
\draw[edge](n1) to[in=120,out=60,loop,style={min distance=7mm}] node[above] {\footnotesize{$(0,1),(3,2)$}} (n1);
\draw[edge,red,very thick](n2) to[in=80,out=20, loop, style={min distance=10mm}] node[right] {\footnotesize{${\bf (0,2)}\color{black}{, (3,3)}$}} (n2);
\draw[edge, red, bend right=10,very thick](n0) edge node[below] {\footnotesize{$\bf (2,2)$}} (n1);
\draw[edge, red, bend right=10,very thick](n1) edge node[above] {\footnotesize{$\bf (1,0)$}} (n0);
\draw[edge, red, bend left=10,very thick](n2) edge node[below] {\footnotesize{$\bf (1,1)$}} (n1);
\draw[edge, red, bend left=10,very thick](n1) edge node[above] {\footnotesize{$\bf (2,3)$}} (n2);
\draw[edge, bend left=50](n0) edge node[above] {\footnotesize{$(1,3)$}} (n2);
\draw[edge, bend right=90](n2) edge node[above] {\footnotesize{$(2,0)$}} (n0);
\end{tikzpicture}
\captionof{figure}{The $(3,4)$-Hoey-Sloane graph with the labeled edges of $\Gamma_C$ featured in bold red.}
\label{3_4_hsg}
\end{center}

The ordered multiset of the digits of $p$ is $\{d_0,d_1,d_2,d_3,d_4,d_5\}=\{0,1,1,2,2,3\}=\{\overline{0},\overline{1},\overline{1},\overline{2},\overline{2},\overline{3}\}$, and clearly $\overline{d}_j=d_{\rho(j)}$ for all $0 \leq j \leq 5$, where $\rho$ is the reversal permutation, $\rho(j)=5-j$.
To apply Corollary \ref{sym_rev}, we choose
$\pi=\left(
\begin{array}{cccccc}
0 & 1 & 2 & 3 & 4 & 5\\
4 & 3 & 0 & 2 & 1 & 5\\
\end{array}
\right)$ and $\sigma=\left(
\begin{array}{cccccc}
0 & 1 & 2 & 3 & 4 & 5\\
0 & 5 & 1 & 3 & 2 & 4\\
\end{array}
\right)$ so that
\begin{align*}
p&=(d_{\pi(5)},d_{\pi(4)},d_{\pi(3)},d_{\pi(2)},d_{\pi(1)},d_{\pi(0)})_4\\
&=(d_5,d_1,d_2,d_0,d_3,d_4)_4\\
&=(3,1,1,0,2,2)_4\\
&=3\cdot(1,0,1,2,3,2)_4\\
&=3\cdot(d_1,d_0,d_2,d_3,d_5,d_4)_4\\
&=3\cdot(d_{\pi\sigma(5)},d_{\pi\sigma(4)},d_{\pi\sigma(3)},d_{\pi\sigma(2)},d_{\pi\sigma(1)},d_{\pi\sigma(0)})_4.
\end{align*}
Now, since $c_2=2$, Corollary \ref{reflective_sibling_thm} gives us the reflective sibling
\begin{align*}
\overline{p}_{\psi^2}&=(\overline{d}_{\pi\psi^2(5)},\overline{d}_{\pi\psi^2(4)},\overline{d}_{\pi\psi^2(3)},\overline{d}_{\pi\psi^2(2)},\overline{d}_{\pi\psi^2(1)},\overline{d}_{\pi\psi^2(0)})_4\\
&=(\overline{d}_3,\overline{d}_4,\overline{d}_5,\overline{d}_1,\overline{d}_2,\overline{d}_0)_4\\
&=(\overline{2},\overline{2},\overline{3},\overline{1},\overline{1},\overline{0})_4\\
&=(1,1,0,2,2,3)_4.
\end{align*}
We now have everything we need to apply and verify Corollary \ref{sym_rev}:
\begin{align*}
\overline{p}_{\psi^2}&=(d_{\rho\pi\psi^2(5)},d_{\rho\pi\psi^2(4)},d_{\rho\pi\psi^2(3)},d_{\rho\pi\psi^2(2)},d_{\rho\pi\psi^2(1)},d_{\rho\pi\psi^2(0)})_4\\
&=(d_2,d_1,d_0,d_4,d_3,d_5)_4\\
&=(1,1,0,2,2,3)_4\\
&=3\cdot(0,1,2,3,2,1)_4\\
&=3\cdot(d_0,d_1,d_4,d_5,d_3,d_2)_4\\
&=(d_{\rho\pi\sigma\psi^2(5)},d_{\rho\pi\sigma\psi^2(4)},d_{\rho\pi\sigma\psi^2(3)},d_{\rho\pi\sigma\psi^2(2)},d_{\rho\pi\sigma\psi^2(1)},d_{\rho\pi\sigma\psi^2(0)})_4.
\end{align*}
Applying the above to $c_3=2$, we see that the reflective sibling $\overline{p}_{\psi^3}$ is simply $p$.

Expressing the above in terms of the action of $D_6$ on $\mathscr{W}_6$, the representation of the $L$-walk of $p$ and its reflection (not an $L$-walk) is shown in Figure \ref{3_4_L_walk}. A clockwise rotation by two vertices of the graph on the right gives us the representation of the $L$-walk of the reflective sibling, $\overline{p}_{\psi^2}$, of $p$.
\begin{center}
\begin{tikzpicture}
\tikzset{edge/.style = {->,> = latex'}}
\tikzset{vertex/.style = {shape=circle,draw,minimum size=1.5em}} [xscale=2, yscale=2, auto=left,every node/.style={circle,fill=blue!20}]
\node[vertex, initial right] (n0) at (4,5) {$0$};
\node[vertex] (n1) at (3.00000153205039,6.732049923038268) {$1$}; \node[vertex] (n2) at (1.000003064103128,6.732052576626029) {$2$}; \node[vertex] (n3) at (0.000000000007041,5.000005307179587) {$2$}; \node[vertex] (n4) at (0.999993871803133,3.2679527305616887) {$1$}; \node[vertex] (n5) at (2.999992339736313,3.2679447697984068) {$0$};
\draw[edge] (n0) edge node[right] {$(2,2)$} (n1);
\draw[edge] (n1) edge node[above] {$(2,3)$} (n2);
\draw[edge] (n2) edge node[left] {$(0,2)$} (n3);
\draw[edge] (n3) edge node[left] {$(1,1)$} (n4);
\draw[edge] (n4) edge node[below] {$(1,0)$} (n5);
\draw[edge] (n5) edge node[right] {$(3,1)$} (n0);
\node[vertex, initial right] (n0) at (11.2,5) {$2$};
\node[vertex] (n1) at (10.20000153205039,6.732049923038268) {$1$}; \node[vertex] (n2) at (8.200003064103128,6.732052576626029) {$0$}; \node[vertex] (n3) at (7.200000000007041,5.000005307179587) {$0$}; \node[vertex] (n4) at (8.199993871803133,3.2679527305616887) {$1$}; \node[vertex] (n5) at (10.199992339736313,3.2679447697984068) {$2$};
\draw[edge] (n0) edge node[right] {$(1,1)$} (n1);
\draw[edge] (n1) edge node[above] {$(1,0)$} (n2);
\draw[edge] (n2) edge node[left] {$(3,1)$} (n3);
\draw[edge] (n3) edge node[left] {$(2,2)$} (n4);
\draw[edge] (n4) edge node[below] {$(2,3)$} (n5);
\draw[edge] (n5) edge node[right] {$(0,2)$} (n0);
\end{tikzpicture}
\captionof{figure}{The $L$-walk of $p$ (left) and its reflection (right).}
\label{3_4_L_walk}
\end{center}

An interesting feature of the above example is that the reflection of $w$ is also a rotation of $w$, which we may directly observe in Figure \ref{3_4_L_walk}.
\end{example}

\begin{remark}
Excluding examples which involve a leading zero digit, we may obtain the collection of $(3,4)$-permutiples by multiplying the elements of \seqnum{A023060} by $3$. Consequently, the number featured in the new example obtained above, that is $(0,1,2,3,2,1)_4=(1,2,3,2,1)_4$, does not appear as an element of \seqnum{A023060}.
\end{remark}

\subsection{Other symmetries involving dihedral siblings}

In other work \cite{holt_5}, we have observed cases where transposing edge labels is possible by maintaining the same state-transition sequence. We now investigate this question more generally, and begin with a definition.
\begin{definition}
Let $p=(d_k,\ldots, d_0)_b=n\cdot(d_{\sigma(k)},\ldots, d_{\sigma(0)})_b$, be an $(n,b,\sigma)$-permutiple and let $s=(d_{0},d_{\sigma(0)})(d_{1},d_{\sigma(1)})\cdots(d_{k-1},d_{\sigma(k-1)})(d_{k},d_{\sigma(k)})$ be its corresponding permutiple string. A permutation, $\varphi$, of the inputs of $s$ which results in another permutiple string is called a {\it symmetry} of $p$.
\end{definition}

Under this definition, all rotational siblings are the result of applying a symmetry of the form $\psi^j$ to the $L$-walk corresponding to $p$.

We recall that an edge label, $(d_1,d_2)$, cannot appear on distinct edges of $\Gamma$ \cite[Theorem 5]{holt_6}, which makes explicit a basic fact.

\begin{corollary}\label{unique_transition_1}
Let $(d_k,\ldots, d_0)_b=n\cdot(d_{\sigma(k)},\ldots, d_{\sigma(0)})_b$, be an $(n,b,\sigma)$-permutiple, and let $s=(d_{0},d_{\sigma(0)})(d_{1},d_{\sigma(1)})\cdots(d_{k-1},d_{\sigma(k-1)})(d_{k},d_{\sigma(k)})$ be its corresponding permutiple string. Then, $s$ uniquely defines an $L$-walk on $\Gamma$.
\end{corollary}

In terms of symmetries, we have another result.

\begin{corollary}\label{unique_transition_2}
Let $p=(d_k,\ldots, d_0)_b=n \cdot (d_{\sigma(k)},\ldots, d_{\sigma(0)})_b$ be an $(n,b,\sigma)$-permutiple, let $S= \{(c_{0},c_{1}),(c_{1},c_{2}),\ldots, (c_{k-1},c_{k}),(c_{k},c_{0})\}$ be its state-transition sequence, and let
\[
s=(d_{0},d_{\sigma(0)})(d_{1},d_{\sigma(1)})\cdots(d_{k-1},d_{\sigma(k-1)})(d_{k},d_{\sigma(k)})
\]
be its corresponding permutiple string. If $\varphi$ is a symmetry of $p$, then the state-transition sequence of the permutiple string
\[
s_{\varphi}=(d_{\varphi(0)},d_{\sigma\varphi(0)})(d_{\varphi(1)},d_{\sigma\varphi(1)})\cdots(d_{\varphi(k-1)},d_{\sigma\varphi(k-1)})(d_{\varphi(k)},d_{\sigma\varphi(k)})
\]
is $S_{\varphi}=\{(c_{\varphi(0)},c_{\varphi(1)}),(c_{\varphi(1)},c_{\varphi(2)}),\ldots, (c_{\varphi(k-1)},c_{\varphi(k)}),(c_{\varphi(k)},c_{\varphi(0)})\}$.
\end{corollary}

For a permutiple string, $s$, we see that transposing inputs, $(d_i,d_{\sigma(i)})$
and $(d_j,d_{\sigma(j)})$, for which their associated state transitions, $(c_i,c_{i+1})$ and $(c_j,c_{j+1})$, are equal does not change the path traversed on $\Gamma$. That is, we may transpose $(d_i,d_{\sigma(i)})$
and $(d_j,d_{\sigma(j)})$ in $s$ and still have a permutiple string. Furthermore, every symmetry, $\varphi$, which leaves $S$ fixed also results in a permutiple string,
\[
s_\varphi=(d_{\varphi(0)},d_{\sigma\varphi(0)})(d_{\varphi(1)},d_{\sigma\varphi(1)})\cdots(d_{\varphi(k-1)},d_{\sigma\varphi(k-1)})(d_{k},d_{\sigma\varphi(k)}),
\]
yielding an $(n,b,\varphi^{-1} \sigma \varphi)$-permutiple,
\[
p_{\varphi}=(d_{\varphi(k)},\ldots,d_{\varphi(0)})_b=(d_{\sigma\varphi(k)},\ldots,d_{\sigma\varphi(0)})_b.
\]

The next result details the interaction between the symmetries of $p$ which fix $S$, and the dihedral siblings of $p$.

\begin{theorem}\label{symmetries_of_dihedral_siblings}
Let $p=(d_k,\ldots,d_0)_b=(d_{\sigma(k)},\ldots,d_{\sigma(0)})_b$ be an $(n,b,\sigma)$-permutiple with carries $c_k,\ldots,c_0=0$, and let $S$ be the state-transition sequence of the $L$-walk of $p$. Also, let $\varphi$ be a symmetry of $p$ which fixes $S$, and let $\psi$ be the $(k+1)$-cycle $(0,1,\ldots,k)$. If $p_{\psi^j}$ is a rotational sibling of $p$, then $\varphi \psi^j$ is also a symmetry of $p$. In particular,
\[
p_{\varphi\psi^j}=(d_{\varphi\psi^j(k)},\ldots,d_{\varphi\psi^j(0)})_b=n\cdot(d_{\sigma\varphi\psi^j(k)},\ldots,d_{\sigma\varphi\psi^j(0)})_b
\]
is an $(n,b,\psi^{-j}\varphi^{-1} \sigma \varphi\psi^j)$-permutiple. If $\overline{p}_{\psi^j}$ is a reflective sibling of $p$, then
\[
\overline{p}_{\varphi\psi^j}=(\overline{d}_{\varphi\psi^j(k)},\ldots,\overline{d}_{\varphi\psi^j(0)})_b=n\cdot(\overline{d}_{\sigma\varphi\psi^j(k)},\ldots,\overline{d}_{\sigma\varphi\psi^j(0)})_b
\]
is an $(n,b,\psi^{-j}\varphi^{-1} \sigma \varphi\psi^j)$-permutiple.
\end{theorem}

\begin{proof}
Let $S=\{(c_{0},c_{1}),(c_{1},c_{2}),\ldots, (c_{k-1},c_{k}),(c_{k},c_{0}) \}$ be the state-transition sequence of the $L$-walk of $p$, and let
\[
s=(d_{0},d_{\sigma(0)})(d_{1},d_{\sigma(1)})\cdots(d_{k-1},d_{\sigma(k-1)})(d_{k},d_{\sigma(k)})
\]
be its corresponding permutiple string. Applying $\varphi$ to $s$ results in the same walk traversed on $\Gamma$. Thus, since $p_{\psi^j}$ is a rotational sibling of $p$, applying $\psi^j$ to $s$ also results in an $L$-walk. Thus, applying $\varphi$ and $\psi^j$ to the inputs of $s$ results in another permutiple string,
\[
s_{\varphi\psi^j}=(d_{\varphi\psi^j(0)},d_{\sigma\varphi\psi^j(0)})(d_{\varphi\psi^j(1)},d_{\sigma\varphi\psi^j(1)})\cdots(d_{\varphi\psi^j(k)},d_{\sigma\varphi\psi^j(k)}).
\]

Applying a reflection to each digit of $p$, we obtain a walk on $\Gamma$ with state-transition sequence $\overline{S}=\{(\overline{c}_{0},\overline{c}_{1}),(\overline{c}_{1},\overline{c}_{2}),\ldots, (\overline{c}_{k-1},\overline{c}_{k}),(\overline{c}_{k},\overline{c}_{0})\}$,  which is not an $L$-walk since $\overline{c}_{0}=\overline{0}=n-1$. Since $\varphi$ fixes the elements of $S$, it also fixes the elements of $\overline{S}$, so that $\overline{S}_{\varphi}=\{(\overline{c}_{\varphi(0)},\overline{c}_{\varphi(1)}),(\overline{c}_{\varphi(1)},\overline{c}_{\varphi(2)}),\ldots, (\overline{c}_{\varphi(k-1)},\overline{c}_{\varphi(k)}),(\overline{c}_{\varphi(k)},\overline{c}_{\varphi(0)})\}=\overline{S}$, which, again, is not an $L$-walk. However, since $\overline{p}_{\psi^j}$ is a reflective sibling of $p$, we have $\overline{c}_j=\overline{n-1}=0$. It follows that
\[
\overline{S}_{\varphi\psi^j}=\{(\overline{c}_{\varphi\psi^j(0)},\overline{c}_{\varphi\psi^j(1)}),(\overline{c}_{\varphi\psi^j(1)},\overline{c}_{\varphi\psi^j(2)}),\ldots, (\overline{c}_{\varphi\psi^j(k)},\overline{c}_{\varphi\psi^j(0)})\}
\]
is an $L$-walk, from which we know that
\[
\overline{s}_{\varphi\psi^j}=(\overline{d}_{\varphi\psi^j(0)},\overline{d}_{\sigma\varphi\psi^j(0)})(\overline{d}_{\varphi\psi^j(1)},\overline{d}_{\sigma\varphi\psi^j(1)})\cdots(\overline{d}_{\varphi\psi^j(k)},\overline{d}_{\sigma\varphi\psi^j(k)})
\]
is a permutiple string.
\end{proof}

We return to the class, $C$, of permutiples considered in Example \ref{example_0}.

\begin{example}\label{example_6}
Consider the $(4,10)$-permutiple
\[
p=(7,2,7,1,1,9,2,8,8)_{10}=4\cdot(1,8,1,7,7,9,8,2,2)_{10}
\]
constructed in Example \ref{example_0}, whose graph is $G_C$ from Example \ref{example_00}. Its permutiple string is $s=(8,2)(8,2)(2,8)(9,9)(1,7)(1,7)(7,1)(2,8)(7,1)$, and, using the $(4,10)$-Hoey-Sloane graph in Figure \ref{4_10_hsg}, the corresponding state-transition sequence is
\[
S=\{(0,0),(0,0),(0,3),(3,3),(3,3),(3,3),(3,0),(0,3),(3,0)\}.
\]
Using $S_j$ to denote the $j^{th}$ element of $S$, with indexing beginning at $j=0$, we see that $S_0=S_1$, $S_3=S_4=S_5$, $S_2=S_7$, and $S_6=S_8$. Thus, the composition  of any of the transpositions $(0,1)$, $(2,7)$, and $(6,8)$, as well as any permutation which fixes the set $\{3,4,5\}$, is a symmetry which fixes $S$.

Most of the above symmetries trivially result in the same permutiple string. On the other hand, we may consider only those symmetries which fix $S$, but permute the inputs of $s$ in a nontrivial way. Those symmetries are precisely the ones which permute elements of $S$ corresponding to distinct edge labels, such as $S_3=(3,3)$ and $S_4=(3,3)$, whose corresponding edge labels are $(9,9)$ and $(1,7)$, respectively. Thus, the transposition $\varphi_1=(3,4)$ is a symmetry yielding a new permutiple string, $s_{\varphi_1}=(8,2)(8,2)(2,8)(1,7)(9,9)(1,7)(7,1)(2,8)(7,1)$. We see that $\varphi_2=(3,5)$ is another symmetry of $p$ which gives us the permutiple string $s_{\varphi_2}=(8,2)(8,2)(2,8)(1,7)(1,7)(9,9)(7,1)(2,8)(7,1)$.  It is clear that $\varphi_1$ and $\varphi_2$ are the only nontrivial symmetries which fix the state-transition sequence (in the sense that they produce permutiple strings which are distinct from $s$ in the usual set-theoretical context).

We may now use Theorem \ref{symmetries_of_dihedral_siblings} to find the dihedral siblings of the three examples produced above,
\begin{align*}
p&=(7,2,7,1,1,9,2,8,8)_{10}=4\cdot(1,8,1,7,7,9,8,2,2)_{10},\\
p_{\varphi_1}&=(7,2,7,1,9,1,2,8,8)_{10}=4\cdot(1,8,1,7,9,7,8,2,2)_{10},\\
p_{\varphi_2}&=(7,2,7,9,1,1,2,8,8)_{10}=4\cdot(1,8,1,9,7,7,8,2,2)_{10},
\end{align*}
by composing $\varphi_1$ and $\varphi_2$ with the symmetries $\psi$, $\psi^2$, and $\psi^7$ to produce their rotational siblings. The rotational siblings of $p$ are the following:
\begin{align*}
p_{\psi}&=(8,7,2,7,1,1,9,2,8)_{10}=4\cdot(2,1,8,1,7,7,9,8,2)_{10},\\
p_{\psi^2}&=(8,8,7,2,7,1,1,9,2)_{10}=4\cdot(2,2,1,8,1,7,7,9,8)_{10},\\
p_{\psi^7}&=(7,1,1,9,2,8,8,7,2)_{10}=4\cdot(1,7,7,9,8,2,2,1,8)_{10}.
\end{align*}
Applying the rotations to $p_{\varphi_1}$ and $p_{\varphi_2}$, we have the following:
\begin{align*}
p_{\varphi_1\psi}&=(8,7,2,7,1,9,1,2,8)_{10}=4\cdot(2,1,8,1,7,9,7,8,2)_{10},\\
p_{\varphi_1\psi^2}&=(8,8,7,2,7,1,9,1,2)_{10}=4\cdot(2,2,1,8,1,7,9,7,8)_{10},\\
p_{\varphi_1\psi^7}&=(7,1,9,1,2,8,8,7,2)_{10}=4\cdot(1,7,9,7,8,2,2,1,8)_{10},\\
p_{\varphi_2\psi}&=(8,7,2,7,9,1,1,2,8)_{10}=4\cdot(2,1,8,1,9,7,7,8,2)_{10},\\
p_{\varphi_2\psi^2}&=(8,8,7,2,7,9,1,1,2)_{10}=4\cdot(2,2,1,8,1,9,7,7,8)_{10},\\
p_{\varphi_2\psi^7}&=(7,9,1,1,2,8,8,7,2)_{10}=4\cdot(1,9,7,7,8,2,2,1,8)_{10}.
\end{align*}

Here, we draw attention to the fact that all of the examples above have the state-transition sequence $S$, $S_{\varphi_1}$, or $S_{\varphi_2}$. That said, it is not difficult to find other symmetries of $p$ which correspond to permutiple strings distinct from those listed above. We will have more to say about this later. For now, we consider the reflective siblings of $p$. Since $c_j=n-1=3$ for $j=3,4,5$, and $6$, we know, by Corollary \ref{reflective_sibling_thm}, that $p$ has four reflective siblings,
\begin{align*}
\overline{p}_{\psi^3}&=(7,1,1,2,7,2,8,8,0)_{10}=4\cdot(1,7,7,8,1,8,2,2,0)_{10},\\
\overline{p}_{\psi^4}&=(0,7,1,1,2,7,2,8,8)_{10}=4\cdot(0,1,7,7,8,1,8,2,2)_{10},\\
\overline{p}_{\psi^5}&=(8,0,7,1,1,2,7,2,8)_{10}=4\cdot(2,0,1,7,7,8,1,8,2)_{10},\\
\overline{p}_{\psi^6}&=(8,8,0,7,1,1,2,7,2)_{10}=4\cdot(2,2,0,1,7,7,8,1,8)_{10}.
\end{align*}
Applying  $\varphi_1$ and $\varphi_2$ to the reflection of $p$, that is, $\overline{p}=(2,7,2,8,8,0,7,1,1)_{10}$, gives $\overline{p}_{\varphi_1}=(2,7,2,8,0,8,7,1,1)_{10}$ and $\overline{p}_{\varphi_2}=(2,7,2,0,8,8,7,1,1)_{10}$. The reader should note that these are not permutiples. However, applying rotations to $\overline{p}_{\varphi_1}$ and $\overline{p}_{\varphi_2}$, we do obtain new examples:
\begin{align*}
\overline{p}_{\varphi_1\psi^3}&=(7,1,1,2,7,2,8,0,8)_{10}=4\cdot(1,7,7,8,1,8,2,0,2)_{10},\\
\overline{p}_{\varphi_1\psi^4}&=(8,7,1,1,2,7,2,8,0)_{10}=4\cdot(2,1,7,7,8,1,8,2,0)_{10},\\
\overline{p}_{\varphi_1\psi^5}&=(0,8,7,1,1,2,7,2,8)_{10}=4\cdot(0,2,1,7,7,8,1,8,2)_{10},\\
\overline{p}_{\varphi_1\psi^6}&=(8,0,8,7,1,1,2,7,2)_{10}=4\cdot(2,0,2,1,7,7,8,1,8)_{10},\\
\overline{p}_{\varphi_2\psi^3}&=(7,1,1,2,7,2,0,8,8)_{10}=4\cdot(1,7,7,8,1,8,0,2,2)_{10},\\
\overline{p}_{\varphi_2\psi^4}&=(8,7,1,1,2,7,2,0,8)_{10}=4\cdot(2,1,7,7,8,1,8,0,2)_{10},\\
\overline{p}_{\varphi_2\psi^5}&=(8,8,7,1,1,2,7,2,0)_{10}=4\cdot(2,2,1,7,7,8,1,8,0)_{10},\\
\overline{p}_{\varphi_2\psi^6}&=(0,8,8,7,1,1,2,7,2)_{10}=4\cdot(0,2,2,1,7,7,8,1,8)_{10}.
\end{align*}

As mentioned above, it is not difficult to find symmetries of $p$ which do not correspond to a cyclic permutation of $S$. For example, taking
$\varphi=\left(
\begin{array}{cccccccccc}
0 & 1 & 2 & 3 & 4 & 5 & 6 & 7 & 8\\
0 & 2 & 3 & 4 & 5 & 6 & 1 & 7 & 8\\
\end{array}
\right)$, we obtain the permutiple string $s_{\varphi}=(8,2)(2,8)(9,9)(1,7)(1,7)(7,1)(8,2)(2,8)(7,1)$ corresponding to the state-transition sequence $
S_{\varphi}=\{(0,0),(0,3),(3,3),(3,3),(3,3),(3,0),(0,3),(0,0),(3,0)\}$, which is not a cyclic permutation of the elements of $S$. The reader may verify the above by using the multi-image union, $\Delta_I$, shown in Figure \ref{4_10_mi_union_1}. The resulting permutiple, $p_{\varphi}=(7,2,8,7,1,1,9,2,8)_{10}=4\cdot (1,8,2,1,7,7,9,8,2)_{10}$, is not on the above list. The techniques which we applied to $p$ may now be applied to $p_{\varphi}$. Continuing in this fashion, we may find all $72$ examples belonging to the permutiple class containing $p$ having the same multiset of digits. We leave this task to the ambitious reader.
\end{example}

To summarize, with Examples \ref{example_0} and \ref{example_000} in mind, a multiset union of mother-graph cycles whose corresponding multi-image union, $\Delta_I$, contains the zero state, is strongly connected, and the indegree and outdegree are equal at each vertex, makes manufacturing new permutiple examples of a specified length a straightforward task. If we already have a known example, $p$, in hand, we may easily construct $\Delta_I$ from the permutiple string and carries of $p$. From there, we may find the dihedral siblings of $p$ by using Corollary \ref{reflective_sibling_thm} and Theorem \ref{rotational_sibling_thm}. To $p$ and each of its dihedral siblings, we may then apply Theorem \ref{symmetries_of_dihedral_siblings} to find symmetries which fix the state-transition sequences of the dihedral siblings, but permute edge labels in a nontrivial way, as seen in Example \ref{example_6}. To find even more examples with the same multiset of digits as $p$, we may either reexamine $\Delta_I$,  or find non-cyclic permutations of $S$ which begin and end with the zero state. If $p$ is contained in a permutiple class, $C$, the above considerations, coupled with Corollary \ref{class_eqiv} in the next section, give us everything we need to find the collection of all permutiples in $C$ which have the same multiset of digits as $p$. Additionally, we may find all permutiples in the reflected class, $\overline{C}$, having the same multiset of digits as the reflective siblings of $p$.

The conjugacy class given in Table \ref{conj_class_table} also puts the considerations of this section on direct display; all of the symmetries, $\pi$, are of the form $\varphi\psi^j$, where $\varphi$ fixes the state-transition sequence.

\section{Characterizing permutiple symmetries}

We now relate permutiple symmetries to ideas encountered in previous work \cite{holt_3,holt_4,holt_5}. Theorem \ref{conj_class} tells us that conjugate permutiples must have the same graph, so that by Definition \ref{perm_class}, they must also belong to the same permutiple class. To demonstrate the converse of these statements would require that we make certain assumptions about digit permutations in the presence of repeated digits. To circumvent such tedium, we define a less strict notion of permutiple conjugacy which we call {\it coarse conjugacy}.

\begin{definition} \label{coarse_conj_class_def}
Suppose $(d_k, d_{k-1},\ldots, d_0)_b$ is an $(n,b,\sigma)$-permutiple. Also, suppose $p_1=(d_{\pi_1(k)}, d_{\pi_1(k-1)},\ldots, d_{\pi_1(0)})_b$, is an $(n,b, \tau_1)$-permutiple, and  $p_2=(d_{\pi_2(k)}, d_{\pi_2(k-1)},\ldots, d_{\pi_2(0)})_b$ is an $(n,b, \tau_2)$-permutiple. Then, we say that $p_1$ and $p_2$ are \textit{coarsely conjugate} if $d_{\pi_1 \tau_1 \pi_1^{-1}(j)}=d_{\pi_2 \tau_2 \pi_2^{-1}(j)}$ for all $0 \leq j \leq k$.
\end{definition}

To distinguish coarse conjugacy from other notions, we  refer to permutiple conjugacy, as defined in Definition \ref{conj_class}, as {\it fine} conjugacy. We see that fine and coarse conjugacy are equivalent definitions in the absence of repeated digits. In the case of repeated digits, however, fine conjugacy requires that we treat each repeated digit as distinct from other digits having the same value. That is, under fine conjugacy, two permutiple examples which are numerically equal can be considered distinct based upon how we choose to permute repeated digits \cite[Example 2]{holt_4}. On the other hand, coarse conjugacy does not make this distinction; any two permutiples with distinct permutations, but give the same arrangement of repeated digits, are considered members of the same coarse conjugacy class. That is, within a multiset framework, permutiples are, quite literally, more coarsely partitioned into classes. More precisely, if $\sigma_1$ and $\sigma_2$ are permutations on the set $\{0,1,\ldots,k\}$, and $p=(d_k, d_{k-1},\ldots, d_0)_b= n \cdot (d_{\sigma_1(k)}, d_{\sigma_1(k-1)},\ldots, d_{\sigma_1(0)})_b=n \cdot (d_{\sigma_2(k)}, d_{\sigma_2(k-1)},\ldots, d_{\sigma_2(0)})_b$, then, if $p_1=(d_{\pi(k)}, d_{\pi(k-1)},\ldots, d_{\pi(0)})_b$, is an $(n,b, \tau)$-permutiple which is coarsely conjugate to $p$, we have that $d_{\sigma_1(j)}=d_{\sigma_2(j)}=d_{\pi\tau\pi^{-1}(j)}$ for all $0 \leq j \leq k$. If repeated digits are present, then it is not necessarily the case that $\sigma_1=\sigma_2=\pi\tau\pi^{-1}$, meaning that fine conjugacy depends on the choice of permutation. Coarse conjugacy does not depend on this choice. In this way, we see that coarse conjugacy class membership is more permissive than fine class membership, which further explains the choice of terminology.

The next two results set the stage for a characterization which involves permutiple symmetries. The first of these is a characterization of coarse conjugacy in terms of permutiple graphs.

\begin{theorem}
Let $(d_k, d_{k-1},\ldots, d_0)_b$ be an $(n,b,\sigma)$-permutiple. Also, suppose that $p_1=(d_{\pi_1(k)}, d_{\pi_1(k-1)},\ldots, d_{\pi_1(0)})_b$ is an $(n,b, \tau_1)$-permutiple, and
$p_2=(d_{\pi_2(k)}, d_{\pi_2(k-1)},\ldots, d_{\pi_2(0)})_b$ is an $(n,b, \tau_2)$-permutiple. Then, $G_{p_1}=G_{p_2}$ if and only if $p_1$ and $p_2$ are coarsely conjugate.
\end{theorem}

\begin{proof}
If $G_{p_1}=G_{p_2}$, then $E_{p_1}=E_{p_2}$. Thus,
\[
\{(d_{\pi_1(j)},d_{\pi_1 \tau_1(j)})\mid 0 \leq j \leq k\}=\{(d_{\pi_2(j)},d_{\pi_2 \tau_2(j)}) \mid 0 \leq j \leq k\},
\]
from which we have
\[
\{(d_{j},d_{\pi_1 \tau_1\pi_1^{-1}(j)}) \mid 0 \leq j \leq k\}=\{(d_{j},d_{\pi_2 \tau_2\pi_2^{-1}(j)}) \mid 0 \leq j \leq k\}.
\]
It follows that $d_{\pi_1 \tau_1\pi_1^{-1}(j)}=d_{\pi_2 \tau_2\pi_2^{-1}(j)}$ for all $0 \leq j \leq k$.
The converse is argued similarly.
\end{proof}

The symmetries of a permutiple, $p$, characterize the graph of $p$.

\begin{theorem}
Suppose $p=(d_k, d_{k-1},\ldots, d_0)_b$ is an $(n,b,\sigma)$-permutiple, and suppose $p_{\pi}=(d_{\pi(k)}, d_{\pi(k-1)},\ldots, d_{\pi(0)})_b$ is an $(n,b,\tau)$-permutiple. Then, $\pi$ is a symmetry of $p$ if and only if $G_{p}=G_{p_{\pi}}$.
\end{theorem}

\begin{proof}
Suppose $\pi$ is a symmetry of $p$, and let
\[
s=(d_{0},d_{\sigma(0)})(d_{1},d_{\sigma(1)})\cdots(d_{k-1},d_{\sigma(k-1)})(d_{k},d_{\sigma(k)})
\]
be the permutiple string corresponding to $p$. Also, let
\[
s_{\pi}=(d_{\pi(0)},d_{\tau\pi(0)})(d_{\pi(1)},d_{\tau\pi(1)})\cdots(d_{\pi(k-1)},d_{\tau\pi(k-1)})(d_{\pi(k)},d_{\tau\pi(k)})
\]
be the permutiple string of $p_{\pi}$.
Since $\pi$ is a symmetry of $p$, we have, by definition, that
\[
\widehat{s}_{\pi}=(d_{\pi(0)},d_{\sigma\pi(0)})(d_{\pi(1)},d_{\sigma\pi(1)})\cdots(d_{\pi(k-1)},d_{\sigma\pi(k-1)})(d_{\pi(k)},d_{\sigma\pi(k)})
\]
is a permutiple string. This is to say that
\begin{align*}
 p_{\pi}&=n\cdot (d_{\tau\pi(k)}, d_{\tau\pi(k-1)},\ldots, d_{\tau\pi(0)})_b\\
 &=(d_{\pi(k)}, d_{\pi(k-1)},\ldots, d_{\pi(0)})_b\\
 &=n\cdot (d_{\sigma\pi(k)}, d_{\sigma\pi(k-1)},\ldots, d_{\sigma\pi(0)})_b\\
& =\widehat{p}_{\pi}.
\end{align*}
Hence, the permutiple strings $s$, $\widehat{s}_{\pi}$, and $s_{\pi}$ contain the same collection of inputs. From this fact, it follows that $G_{p}=G_{p_{\pi}}$.

Now suppose that the graphs $G_{p}$ and $G_{p_{\pi}}$ are the same. Since $p$ and $p_{\pi}$ have the same collection of digits, the multiset of inputs in the strings $s$ and $s_{\pi}$ must be identical. Since these are both permutiple strings, $\pi$ must be a symmetry of $p$.
\end{proof}

The above definitions and results give us a characterization of what it means for two permutiples having the same multiset of digits to be members of the same permutiple class.

\begin{corollary}\label{class_eqiv}
Suppose $p=(d_k, d_{k-1},\ldots, d_0)_b$ is an $(n,b,\sigma)$-permutiple, and suppose $p_{\pi}=(d_{\pi(k)}, d_{\pi(k-1)},\ldots, d_{\pi(0)})_b$ is an $(n,b,\tau)$-permutiple. Then, the following statements are equivalent:
\begin{enumerate}
\item $\pi$ is a symmetry of $p$.
\item $G_{p}=G_{p_{\pi}}$.
\item $p$ and $p_{\pi}$ are coarsely conjugate.
\item $p$ and $p_{\pi}$ belong to the same permutiple class.
\end{enumerate}
\end{corollary}

\bigskip
\hrule
\bigskip
\noindent 2020 {\it Mathematics Subject Classification}: Primary 11A63.

\noindent \emph{Keywords:} permutiple,  palintiple, reverse multiple, reverse divisor, digit-preserving multiplication, permutiple class, dihedral sibling, permutiple conjugacy, mother graph, Hoey-Sloane graph, Hoey-Sloane multigraph, Hoey-Sloane machine, Young graph.

\bigskip
\hrule
\bigskip

\noindent (Concerned with sequences
\seqnum{A001232},
\seqnum{A008918},
\seqnum{A008919},
\seqnum{A023059},
\seqnum{A023060},
\seqnum{A023064},
\seqnum{A023086},
\seqnum{A023087},
\seqnum{A023088},
\seqnum{A023089},
\seqnum{A023090},
\seqnum{A023091},
\seqnum{A023092},
\seqnum{A023093},
\seqnum{A031877},
\seqnum{A053654},
\seqnum{A092697},
\seqnum{A096093},
\seqnum{A222814},
\seqnum{A222815}, and
\seqnum{A360518}.)

\bigskip
\hrule
\bigskip

\vspace*{+.1in}
\noindent

\vskip .1in

\end{document}